\documentclass[reqno]{amsart}
\usepackage{amssymb}

\usepackage{verbatim}
\usepackage[hypertex]{hyperref}

\newcommand{\mysection}[1]{\section{#1}
\setcounter{equation}{0}}

\newtheorem{theorem}{Theorem}[section]
\newtheorem{corollary}[theorem]{Corollary}
\newtheorem{lemma}[theorem]{Lemma}
\newtheorem{proposition}[theorem]{Proposition}

\theoremstyle{definition}
\newtheorem{remark}[theorem]{Remark}

\theoremstyle{definition}

\theoremstyle{definition}
\newtheorem{assumption}[theorem]{Assumption}

\makeatletter
\def\dashint{\operatorname%
{\,\,\text{\bf--}\kern-.98em\DOTSI\intop\ilimits@\!\!}}
\makeatother

\def\vu{\textit{\textbf{u}}}
\def\vv{\textit{\textbf{v}}}
\def\vw{\textit{\textbf{w}}}
\def\vf{\textit{\textbf{f}}}

\def\bfA{\mathbf{A}}
\def\bR{\mathbb{R}}
\def\bN{\mathbb{N}}

\def\bH{\mathbb{H}}

\def\bC{\mathbb{C}}

\def\bW{\mathbb{W}}

\def\fL{\mathfrak{L}}

\def\cC{\mathcal{C}}
\def\cD{\mathcal{D}}
\def\cE{\mathcal{E}}

\def\cH{\mathcal{H}}
\def\cP{\mathcal{P}}
\def\cM{\mathcal{M}}
\def\cO{\mathcal{O}}
\def\cQ{\mathcal{Q}}
\def\cR{\mathcal{R}}

\def\cL{\mathcal{L}}

\newcommand{\set}[1]{\left\{#1\right\}}

\newcommand{\ii}{\textit{\textbf{i\,}}}

\begin{document}
\title[Higher order systems]{On the $L_p$-solvability of higher order parabolic and elliptic systems with
BMO coefficients}

\author[H. Dong]{Hongjie Dong}
\address[H. Dong]{Division of Applied Mathematics, Brown University,
182 George Street, Providence, RI 02912, USA}
\email{Hongjie\_Dong@brown.edu}
\thanks{H. Dong was partially supported by NSF grant number DMS-0800129.}

\author[D. Kim]{Doyoon Kim}
\address[D. Kim]{Department of Applied Mathematics, Kyung Hee University, 1, Seochun-dong, Gihung-gu, Yongin-si, Gyeonggi-do 446-701 Korea}
\email{doyoonkim@khu.ac.kr}

\subjclass[2000]{35K52, 35J58}

\keywords{higher order systems, boundary value problems, BMO coefficients, Sobolev spaces.}

\begin{abstract}
We prove the solvability in Sobolev spaces for both divergence and non-divergence form higher order parabolic and elliptic systems in the whole space, on a half space, and on a bounded domain.
The leading coefficients are assumed to be merely measurable in the time variable and have small mean oscillations with respect to the spatial variables in small balls or cylinders. For the proof,
we develop a set of new techniques to produce mean oscillation estimates for systems on a half space.
\end{abstract}

\maketitle

\setcounter{tocdepth}{1}
\tableofcontents

\mysection{Introduction}

The paper is devoted to the study of the $L_p$-theory of higher order parabolic and elliptic systems. 
More precisely, we expand the $L_p$-theory of higher order elliptic and parabolic systems
to include a class of {\em not necessarily continuous coefficients} via a unified approach for both divergence type and non-divergence type systems in the whole space, on a half space, and on a bounded domain.
The coefficients we consider are complex valued and,
especially, the leading coefficients of parabolic systems are only measurable in the time variable and belong to the class of BMO (bounded mean oscillations) as functions of the spatial variables. The mean oscillations of the coefficients only need to be sufficiently small over small cylinders.

To present the exact forms of systems, we let
$$
Lu = \sum_{|\alpha|\le m,|\beta|\le m} A^{\alpha\beta}D^{\alpha}D^{\beta}u,\quad
\cL u = \sum_{|\alpha|\le m,|\beta|\le m}D^\alpha(A^{\alpha\beta}D^\beta u),
$$
where $m$ is a positive integer,
$$
D^\alpha=D_1^{\alpha_1}\cdots D_d^{\alpha_d},
\quad
\alpha=(\alpha_1,\cdots,\alpha_d),
$$
and, for each $\alpha$, $\beta$, $A^{\alpha\beta}=[A^{\alpha\beta}_{ij}(t,x)]_{i,j=1}^n$ is an $n\times n$ complex matrix-valued function.
The involved functions are complex vector-valued functions, that is,
$$
u = (u^1,\cdots,u^n)^{\text{tr}},
\quad
f = (f^1,\cdots,f^n)^{\text{tr}},
\quad
f_{\alpha}=(f_{\alpha}^1,\cdots,f_{\alpha}^n)^{\text{tr}}.
$$
The parabolic systems we study are
$$
u_t + (-1)^m L u = f,
\quad
u_t + (-1)^m \cL u = \sum_{|\alpha|\le m}D^\alpha f_{\alpha},
$$
where the  first one is in non-divergence form and the second one is in divergence form.
The elliptic systems, non-divergence form and divergence form, respectively, are
$$
Lu = f,
\quad
\cL u = \sum_{|\alpha|\le m}D^\alpha f_{\alpha}.
$$
Whenever elliptic systems are considered, coefficients, $u$, $f$, and $f_{\alpha}$ are independent of $t$. When the domain is other than the whole space, we impose the homogeneous Dirichlet boundary condition.

In the case of non-divergence type elliptic systems, we prove that, for a given $f \in L_p(\Omega)$, there is a unique solution $u \in W_p^{2m}(\Omega)$ to the system
$Lu = f$ in $\Omega$, where $\Omega$ is either the whole space $\bR^d$, the half space $\bR^d_+ = \{(x_1,\cdots,x_d) \in \bR^d, x_1 > 0\}$, or a bounded domain.
We also prove the corresponding results for the other types of elliptic and parabolic systems; see Section \ref{sec082001}.

As is well known, the key ingredient in establishing $L_p$-theory is apriori $L_p$-estimates of solutions to given systems. Largely, this is done in two steps.
First, one establishes $L_p$-estimates for systems with `simple' coefficients, for example, constant coefficients.
Second, if the given system is in some sense close to systems with simple coefficients, one obtains the desired $L_p$-estimates by using a perturbation argument.

The $L_p$-estimates for systems with constant coefficients, in many references, for example, \cite{ADN64},  rely on the exact representation of solutions and the Calder\'on-Zygmund theorem.
Another approach for such $L_p$-estimates is that of Campanato-Stampachia
using Stampacchia's interpolation theorem (see \cite{Giaq93}).
As to perturbation arguments, if the coefficients of given systems are uniformly continuous, the estimates are carried out by using the local closeness of the coefficients to constant coefficients in $L_{\infty}$ norm.
When the class of VMO (vanishing mean oscillations) coefficients was first introduced,
another perturbation argument was used in \cite{CFL1,CFL2,BC93},
where the continuity of coefficients is measured in the average sense,
not in the pointwise sense, through a representation formula of solutions and the Coifman-Rochberg-Weiss commutator theorem.

In this paper, in establishing the key $L_p$-estimates, we replace the first step,
$L_p$-estimates of solutions to systems with simple coefficients, by {\em mean oscillation estimates} of solutions to the systems.
Then for the second step we use a different perturbation argument, which is well suited to the mean oscillation estimates.
For instance, if the system under consideration is elliptic in the form of $Lu = f$ with constant coefficients in the whole space,
then by the mean oscillation estimate of $D^{2m}u$ we mean a pointwise estimate of the form
\begin{multline}
								\label{eq0901}
\dashint_{B_r(x_0)}|D^{2m}u - \dashint_{B_r(x_0)} D^{2m}u \, dy | \, dx
\\
\le N\kappa^{-1}\left(\dashint_{B_{\kappa r}(x_0)} |D^{2m} u|^2 \, dx \right)^{1/2}
+N\kappa^{\frac d 2} \left(\dashint_{B_{\kappa r}(x_0)}|f|^2 \, dx \right)^{1/2}
\end{multline}
for all $x_0 \in \bR^d$, $r \in (0,\infty)$, and $\kappa \in [\kappa_0,\infty)$,
where $B_r(x_0)$ is a ball with center $x_0$ and radius $r$.
Indeed, this implies the $L_p$-estimate of $D^{2m}u$ by the well known Fefferman-Stein theorem on sharp functions and the Hardy-Littlewood maximal function theorem.
But more importantly, this type of estimates well embraces the perturbation between the original systems and systems with simple coefficients when the coefficients have small mean oscillations over small balls or small parabolic cylinders.
This approach was first introduced by Krylov \cite{Krylov_2005,Krylov_2007_mixed_VMO} to deal with second order elliptic and parabolic equations with VMO coefficients in the whole space, and is well explained in his book \cite{Krylov:book:2008}.

Due to the well adaptiveness of estimates like \eqref{eq0901} to the perturbation argument, our main effort in this paper focuses on obtaining mean oscillation estimates of systems with simple coefficients.
Since in the parabolic case we allow coefficients to be only measurable in the time direction, the systems with simple coefficients in our case are naturally those with {\em measurable coefficients depending only on $t$}.

For systems in the whole space, which corresponds to interior estimates,
the mean oscillation estimates follow rather easily by adapting the techniques in \cite{Krylov_2005,Krylov:book:2008} to higher order systems.
However, differently from the arguments in \cite{Krylov_2005}, we derive the non-divergence case as a corollary from the divergence case.
Another noteworthy difference is that we prove the mean oscillation estimates not only for the highest order terms but also for the lowest order terms, so we are able to avoid the argument in \cite{Krylov_2005} deriving the $L_p$-estimates of solutions from those of the highest order terms, which is technically difficult in the case of higher order equations.

For systems on a half space or on a bounded domain, which corresponds to boundary estimates, it is not possible to use the approach in \cite{Krylov_2005, Krylov:book:2008} since the estimates developed there are only for equations {\em in the whole space} (interior estimates).
Thus here we develop a set of new techniques to produce mean oscillation estimates
for systems {\em on a half space}.
This is a new approach for boundary  $L_p$-estimates, which is applicable to a wide class of equations or systems.
To get these boundary mean oscillation estimates, as in the whole space case, we start with $L_2$-estimates of systems on a half space.
Although the $L_2$-estimate for divergence type systems is well known under appropriate ellipticity or parabolicity conditions on the leading coefficients,
our Theorem \ref{th06_01} regarding the $L_2$-estimate for non-divergence type systems on a half space with coefficients measurable in time, as it alone, is a new result to our best knowledge.
In the proof we only use that of divergence type systems and an interpolation argument.
It is worth noting that $L_2$-estimates for higher order elliptic equations and systems were obtained in \cite{Fried76,MR1914441} by using bootstrap arguments. For parabolic equations, however, in \cite{Fried76} the coefficients are assumed to be H\"older continuous in the time variable since a semigroup method was used.

From the $L_2$-estimates, we derive the boundary mean oscillation estimates of {\em some} of highest order derivatives of solutions, precisely, $D_{x'}^m u$ in the case of divergence systems and $D_{x'}^{2m}u$ in the case of non-divergence systems,
where $x'$ denotes the last $d-1$ coordinates of $x = (x_1,x')$ in $\bR^d$. These estimates alone, however, are not sufficient for us to prove the main theorems. Because of this, we then consider a parabolic system with special coefficients, such that in a periodic pattern certain order normal derivatives of solutions to the system vanish on the boundary. This gives us
the boundary mean oscillation estimates of $D_{1}^m u$ or $D_{1}^{2m} u$; see Lemma \ref{lem3.34}.
Once we have all required mean oscillation estimates, we proceed as in \cite{Krylov_2005} to the desired $L_p$-estimates using the perturbation argument, the details of which are illustrated for divergence type systems in the whole space; see Section \ref{sec5}.

In the literature, for uniformly continuous coefficients, a rather complete $L_p$-theory can be found for general linear elliptic systems in \cite{ADN64, A65} and for parabolic systems in \cite{Solo,LSU,Ejd,Fried}.
If coefficients are in the class of VMO, non-divergence type higher order systems in the whole space have been investigated, for example, in \cite{CFF,HHH,PS1,PS3}, where leading coefficients of systems are either VMO with respect to all the variables or independent of the time variable. For divergence type higher order elliptic systems with VMO coefficients, we refer the reader to a recent interesting preprint \cite{MaMiSh} in which the inhomogeneous Dirichlet problem on Lipschitz domains was studied. In all these papers, the method of singular integrals is used, so measurable coefficients are not allowed.

Restricted to {\em second order} systems or equations, there are a relatively larger number of papers which can be compared to this paper.
Non-divergence elliptic and parabolic equations on smooth domains with VMO coefficients were first studied in \cite{CFL1,CFL2,BC93} by using the technique of singular integrals. For further related results, we refer the reader to the book \cite{MaPaSo00} and reference therein. The corresponding results for divergence elliptic equations were obtained in \cite{DFG,AuscherQafsaoui} by a similar technique. These results were later improved by the authors of \cite{BW04} in several papers for divergence type equations/systems without lower order terms on non-smooth domains by using a perturbation argument based on the maximal function theorem and a covering lemma (see \cite{B09} for an extension to fourth order systems).
An interesting question would be whether the methods in \cite{BW04,B09} can be applied to equations with lower order terms or non-divergence form equations/systems.
The methodology developed by Krylov in \cite{Krylov_2005, Krylov_2007_mixed_VMO} was later developed and extended in \cite{DK08} for divergence and non-divergence systems in the whole space with the same class of coefficients, and in \cite{KimKrylov07,KimKrylov:par06,Krylov08} for non-divergence parabolic and elliptic equations in the whole space with partially BMO coefficients for $p>2$, and in \cite{Dong09} for any $p\in (1,\infty)$. In \cite{DongKim08a,Dong09,DKCalVarPDE}, this method was further adapted to divergence parabolic and elliptic equations/systems in the whole space with partially BMO coefficients.
It is worth noting that in \cite{Dong09}-\cite{DKCalVarPDE} and \cite{KimKrylov07}-\cite{Krylov08} only interior mean oscillation estimates were derived. When dealing with equations and systems on a half space or on a bounded domain in \cite{KimKrylov07,KimKrylov:par06,DongKim08a,DKCalVarPDE},
the authors took full advantage of the facts that the coefficients are allowed to be merely measurable in one spatial direction and the given systems are second order. Thus without using any {\em boundary mean oscillation estimates} developed here, the boundary $L_p$-estimates were derived from interior estimates as corollary type results by using odd and even extension techniques. However, the extension techniques do not work for higher order equations or systems.  This is the first paper in which the ideas in  \cite{Krylov_2005, Krylov_2007_mixed_VMO} are adapted to boundary estimates, in both divergence and non-divergence cases.

As noted above, the first critical step of the proof is the $L_2$-estimates of
systems with relatively simple coefficients under the ellipticity or parabolicity conditions on the leading coefficients.
In this paper, we use so-called Legendre-Hadamard ellipticity condition, which is more general than the strong ellipticity condition considered, for example, in \cite{LKS,B09,DKCalVarPDE}. Nevertheless, it is still stronger than the uniform parabolicity condition in the sense of Petrovskii, which was used in \cite{Ejd,PS1,Solo} with more regularity assumptions on the leading coefficients.
We shall discuss in details these conditions in Section \ref{sec11}.

The organization of the paper is as follows. We introduce some notation and state the main results in the next section. The remaining part of the paper is divided into two parts. In the first part, we treat systems in the whole space. Section \ref{sec3} and \ref{sec_aux} are devoted to the $L_2$-estimates and mean oscillation estimates for both divergence and non-divergence parabolic systems with simple coefficients. In Section \ref{sec5} we complete the proofs of the $L_p$-solvability of systems in the whole space. The second part is the main part of the paper, in which we treat systems on a half space or on a bounded domain. In Section \ref{sec6} we establish the $L_2$-solvability of divergence and non-divergence parabolic systems with simple coefficients on a half space. Then in Section \ref{sec7}, we obtain the boundary mean oscillation estimates of $D_{x'}^m u$ and $D_{x'}^{2m} u$ for divergence and non-divergence systems respectively. Section \ref{sec8} is devoted to the estimates for a special type of systems. With these preparations, in Section \ref{sec9} and \ref{sec10} we establish the $L_p$-solvability of both divergence and non-divergence parabolic systems on a half space and on a bounded domain. Finally, we discuss in Section \ref{sec11} some other ellipticity conditions used in the literature, and show how our results can be extended to systems under those conditions.

\mysection{Main results}								 \label{sec082001}

We first introduce some notation used throughout the paper.
A point in $\bR^d$ is denoted by $x = (x_1,\cdots,x_d)$.
Whenever needed, we denote $x$ by $(x_1,x')$ where $x' \in \bR^{d-1}$.
A point in
$$
\bR^{d+1} = \bR \times \bR^d = \{ (t,x) : t \in \bR, x \in \bR^d \}
$$
is denoted by $X = (t,x)$.
For $T \in (-\infty,\infty]$, set
$$
\cO_T = (-\infty,T) \times \bR^d,
\quad
\cO_T^+ = (-\infty,T) \times \bR^d_+,
$$
where $\bR^d_+ = \{x = (x_1,\cdots,x_d) \in \bR: x_1 > 0\}$.
Especially, if $T = \infty$, we have, for example,
$\cO_{\infty}^+ = \bR \times \bR^d_+$.
We also have
$$
B_r(x) = \{ y \in \bR^d: |x-y| < r\},
\quad
B_r'(x') = \{ y' \in \bR^{d-1}: |x'-y'| < r\},
$$
$$
Q_r(t,x) = (t-r^{2m},t) \times B_r(x),
\quad
Q_r'(t,x') = (t-r^{2m},t) \times B_r'(x'),
$$
$$
Q_r^+(t,x) = Q_r(t,x) \cap\cO_\infty^+.
$$
We denote
$$
\langle f,g \rangle_{\Omega} = \int_{\Omega} f^{\text{tr}} \bar{g}
= \sum_{j=1}^n\int_{\Omega} f^j \overline{g^j}.
$$
For a function $f$ on $\cD \subset \bR^{d+1}$, we set
\begin{equation*}
(f)_{\cD} = \frac{1}{|\cD|} \int_{\cD} f(t,x) \, dx \, dt
= \dashint_{\cD} f(t,x) \, dx \, dt,
\end{equation*}
where $|\cD|$ is the
$d+1$-dimensional Lebesgue measure of $\cD$.

In order to state and prove our results on systems in Sobolev spaces,
in addition to the well known spaces $L_p$ and $W_p^k$,
we introduce the following function spaces.
As a solution space for non-divergence type parabolic equations,
we use
$$
W_p^{1,2m}((S,T) \times \Omega)
=\{u : u_t, D^{\alpha}u \in L_p((S,T)\times\Omega), 0 \le |\alpha| \le 2m\}
$$
equipped with its natural norm.
Unless specified otherwise,  in this paper $D^{\alpha}u(t,x)$ means the spatial derivative of $u$.
For divergence type parabolic equations with $\Omega = \bR^d$, we introduce
$$
\cH_p^m((S,T) \times \bR^d)
= (1-\Delta)^{\frac m 2}W_p^{1,2m}((S,T) \times \bR^d)
$$
equipped with the norm
$$
\|u\|_{\cH_p^m((S,T) \times \bR^d)}
= \| (1-\Delta)^{-\frac m 2} u \|_{W_p^{1,2m}((S,T)\times\bR^d)}.
$$
Note that if we set
$$
\bH_p^{-m}((S,T) \times \bR^d)
= (1-\Delta)^{\frac m 2}L_p((S,T) \times \bR^d),
$$
$$
\|f\|_{\bH_p^{-m}((S,T) \times \bR^d)}
= \|(1-\Delta)^{-\frac m 2} f\|_{L_p((S,T)\times\bR^d)},
$$
then
$$
\|u\|_{\cH_p^m((S,T) \times \bR^d)}
\cong
\|u_t\|_{\bH_p^{-m}((S,T) \times \bR^d)}
+ \sum_{|\alpha|\le m} \|D^{\alpha}u\|_{L_p((S,T)\times\bR^d)}.
$$
For a general $\Omega$, we set
$$
\bH^{-m}_p((S,T)\times \Omega)
= \left\{ f: f = \sum_{|\alpha|\le m} D^{\alpha}f_{\alpha}, \quad f_{\alpha} \in L_p((S,T) \times \Omega)\right\},
$$
$$
\|f\|_{\bH^{-m}_p((S,T)\times \Omega)}
= \inf \left\{ \sum_{|\alpha|\le m} \|f_{\alpha}\|_{L_p((S,T)\times \Omega)} : f = \sum_{|\alpha|\le m} D^{\alpha}f_{\alpha}\right\},
$$
and
$$
\cH_p^m((S,T) \times \Omega)
=\{u: u_t \in \bH_p^{-m}((S,T)\times\Omega), D^{\alpha}u \in L_p((S,T)\times\Omega), 0 \le |\alpha| \le m \},
$$
$$
\|u\|_{\cH_p^m((S,T) \times \Omega)}
= \|u_t\|_{\bH_p^{-m}((S,T)\times\Omega)} + \sum_{|\alpha|\le m} \|D^{\alpha}u\|_{L_p((S,T)\times\Omega)}.
$$

Let $\delta,K>0$ be two constants. Throughout the paper, we assume that all the coefficients are measurable, complex valued and bounded,
$$
|A^{\alpha\beta}| \le
\left\{
\begin{aligned}
\delta^{-1}, \quad &|\alpha| = |\beta| = m,
\\
K, \quad &\text{otherwise}.
\end{aligned}
\right.
$$
In addition, we impose the Legendre-Hadamard ellipticity on the leading coefficients (see, for instance, \cite{Fried,Giaq93}).
Here we call $A^{\alpha\beta}$ the leading coefficients if $|\alpha|=|\beta|=m$. All the other coefficients are called lower-order coefficients.
By the Legendre-Hadamard ellipticity we mean
\begin{equation}
                                \label{eq7.9.17}
\Re\left(\sum_{|\alpha|=|\beta|=m}
\theta^{\text{tr}} \xi^{\alpha}\xi^{\beta}A^{\alpha\beta}(t,x)\bar\theta\right)
\ge \delta |\xi|^{2m}|\theta|^2
\end{equation}
for all $(t,x) \in \bR^{d+1}$, $\xi \in \bR^d$, and $\theta \in \bC^n$.
Here we use $\Re(f)$ to denote the real part of $f$.

Now we state our regularity assumption on the leading coefficients.
Let
$$
\text{osc}_x\left(A^{\alpha\beta},Q_r(t,x)\right)
= \dashint_{t-r^{2m}}^{\,\,\,t}
\dashint_{B_r(x)} \big| A^{\alpha\beta}(s,y) - \dashint_{B_r(x)} A^{\alpha\beta}(s,z) \, dz \big| \, dy \, ds.
$$
Then we set
$$
A^{\#}_R = \sup_{(t,x) \in \bR^{d+1}} \sup_{r \le R}  \sup_{|\alpha|=|\beta|=m} \text{osc}_{x} \left(A^{\alpha\beta},Q_r(t,x)\right).
$$

We impose on the leading coefficients the small mean oscillation
condition with a parameter $\rho>0$, which will be
specified later.
\begin{assumption}[$\rho$]                          \label{assumption20080424}
There is a constant $R_0\in (0,1]$ such that $A_{R_0}^{\#} \le \rho$.
\end{assumption}

Contrary to non-divergence type systems where equations are defined almost everywhere,
solutions to divergence type equations are understood in the weak sense.
More precisely, for example,
we say that $u \in \cH_{p,\text{loc}}^m((S,T) \times \Omega)$, where $1<p<\infty$, $\Omega \subset \bR^d$, and $-\infty \le S < T \le \infty$, satisfies
$$
u_t + (-1)^m \cL u + \lambda u = \sum_{|\alpha|\le m}D^\alpha f_{\alpha}
\quad
\text{in}
\quad
(S,T) \times \Omega,
$$
provided that
$$
\int_S^t \int_{\Omega} \left(- \varphi_t \cdot u  + (-1)^{m+|\alpha|} D^{\alpha}\varphi \cdot A^{\alpha\beta} D^{\beta}u \right) \, dx \, ds
$$
$$
= (-1)^{|\alpha|} \int_S^t \int_{\Omega} D^{\alpha}\varphi \cdot f_{\alpha}\, dx \, ds
+ \int_{\Omega} u(S,x) \varphi(S,x) \, dx - \int_{\Omega} u(t,x) \varphi(t,x) \, dx
$$
for every $t \in (S,T]$ and $\varphi = (\varphi^1,\cdots, \varphi^n) \in C^{\infty}(\overline{(S,T)\times\Omega})$
such that $\varphi(t,\cdot) \in C_0^{\infty}(\Omega)$ for all $t \in [S,T]$.
If $S=-\infty$ or $T=\infty$, we take $\varphi \in C^{\infty}(\overline{(S,T)\times\Omega})$
such that $\varphi(-\infty,\cdot) = 0$ or $\varphi(\infty,\cdot) = 0$, respectively.

We are now ready to present our main results.

\begin{theorem}[Divergence parabolic systems in the whole space]
             \label{Thm1}

Let $p \in (1,\infty)$, $T\in (-\infty,\infty]$
and $f_\alpha \in L_p(\cO_T)$ for $|\alpha|\le m$.
Then there exists a constant $\rho=\rho(d,m,n,p,\delta)$
such that, under Assumption \ref{assumption20080424} ($\rho$),
the following hold true.

\noindent
(i)
For any $u \in \cH_p^m(\cO_T)$ satisfying
\begin{equation}							 \label{eq081902}
u_t+(-1)^m \cL u + \lambda u = \sum_{|\alpha|\le m}D^\alpha f_{\alpha}\quad \text{in}\,\,\cO_T,
\end{equation}
we have
\begin{equation*}							 
\sum_{|\alpha|\le m} \lambda^{1-\frac{|\alpha|}{2m}} \| D^{\alpha} u \|_{L_p(\cO_T)}
\le N \sum_{|\alpha| \le m} \lambda^{\frac{|\alpha|}{2m}} \| f_{\alpha} \|_{L_p(\cO_T)},	
\end{equation*}
provided that $\lambda \ge \lambda_0$,
where $N$ and $\lambda_0 \ge 0$
depend only on $d$, $m$, $n$, $p$, $\delta$, $K$ and $R_0$.

\noindent
(ii)
For any  $\lambda > \lambda_0$, there exists a unique $u \in \cH_p^m(\cO_T)$ satisfying \eqref{eq081902}.
\end{theorem}

\begin{theorem}[Non-divergence parabolic systems  in the whole space]
              \label{Thm2}
Let $p \in (1,\infty)$, $T\in (-\infty,\infty]$
and $f \in L_p(\cO_T)$.
Then there exists a constant $\rho=\rho(d,m,n,p,\delta)$
such that, under Assumption \ref{assumption20080424} ($\rho$),
the following hold true.

\noindent
(i)
For any $u \in W_p^{1,2m}(\cO_T)$ satisfying
\begin{equation}							 \label{eq081902b}
u_t+(-1)^m L u + \lambda u = f \quad \text{in}\,\,\cO_T,
\end{equation}
we have
\begin{equation*}							
\| u_t\|_{L_p(\cO_T)}+\sum_{|\alpha|\le 2m} \lambda^{1-\frac{|\alpha|}{2m}} \| D^{\alpha} u \|_{L_p(\cO_T)}
\le N \| f\|_{L_p(\cO_T)},	
\end{equation*}
provided that $\lambda \ge \lambda_0$,
where $N$ and $\lambda_0 \ge 0$
depend only on $d$, $m$, $n$, $p$, $\delta$, $K$ and $R_0$.

\noindent
(ii)
For any  $\lambda > \lambda_0$, there exists a unique $u \in W_p^{1,2m}(\cO_T)$ satisfying \eqref{eq081902b}.
\end{theorem}

\begin{remark}
								\label{remark02}
We can also solve Cauchy problems for systems defined on $(0,T)\times\bR^d$ in divergence or non-divergence form.
If the initial condition is zero,
this is done by extending the original system to a system defined on $(-\infty, T)\times\bR^d$
with the right-hand side being zero for $t \in (-\infty,0)$.
We deal with, in the same manner, Cauchy problems for the systems below defined on a half space or on a bounded domain.
Note that in the case $T<\infty$, by considering $e^{-(\lambda_0+1) t}u$ instead of $u$ we can take $\lambda=0$ in the theorems above and below with the expense that $N$ also depends on $T$.
\end{remark}

The next two theorems are about the boundary value problem of systems in divergence and non-divergence form on a half space $\cO_T^+ = (-\infty,T) \times \bR^d_+$.

\begin{theorem}[Divergence parabolic systems on a half space]
              \label{Thm5}
Let $p \in (1,\infty)$, $T\in (-\infty,\infty]$
and $f_\alpha \in L_p(\cO_T^+)$ for $|\alpha|\le m$.
Then there exists a constant $\rho=\rho(d,m,n,p,\delta)$
such that, under Assumption \ref{assumption20080424} ($\rho$),
the following hold true.

\noindent
(i)
For any $u \in \cH_p^m(\cO_T^+)$ satisfying
\begin{equation}							 \label{eq081902c}
\left\{
  \begin{aligned}
    u_t+(-1)^m \cL u + \lambda u = \sum_{|\alpha|\le m}D^\alpha f_{\alpha} \quad & \hbox{in $\cO_T^+$;} \\
    u=D_1 u=...=D_1^{m-1}u=0 \quad & \hbox{on $\partial_p \cO_T^+$,}
  \end{aligned}
\right.
\end{equation}
where $\partial_p \cO_T^+ = (-\infty,T) \times \partial \bR^d_+$,
we have
\begin{equation}							 
\sum_{|\alpha|\le m} \lambda^{1-\frac{|\alpha|}{2m}} \| D^{\alpha} u \|_{L_p(\cO_T^+)}
\le N \sum_{|\alpha| \le m} \lambda^{\frac{|\alpha|}{2m}} \| f_{\alpha} \|_{L_p(\cO_T^+)},	
\end{equation}
provided that $\lambda \ge \lambda_0$,
where $N$ and $\lambda_0 \ge 0$
depend only on $d$, $m$, $n$, $p$, $\delta$, $K$ and $R_0$.

\noindent
(ii)
For any  $\lambda > \lambda_0$, there exists a unique $u \in \cH_p^m(\cO_T^+)$ satisfying \eqref{eq081902c}.
\end{theorem}

\begin{theorem}[Non-divergence parabolic systems on a half space]
              \label{Thm6}
Let $p \in (1,\infty)$, $T\in (-\infty,\infty]$
and $f \in L_p(\cO_T^+)$.
Then there exists a constant $\rho=\rho(d,m,n,p,\delta)$
such that, under Assumption \ref{assumption20080424} ($\rho$),
the following hold true.

\noindent
(i)
For any $u \in W_p^{1,2m}(\cO_T^+)$ satisfying
\begin{equation}							 \label{eq081902d}
\left\{
  \begin{aligned}
    u_t+(-1)^m L u + \lambda u = f \quad & \hbox{in $\cO_T^+$;} \\
    u=D_1 u=...=D_1^{m-1}u=0 \quad & \hbox{on $\partial_p \cO_T^+$,}
  \end{aligned}
\right.
\end{equation}
we have
\begin{equation*}							 
\|u_t\|_{L_p(\cO_T^+)}+\sum_{|\alpha|\le 2m} \lambda^{1-\frac{|\alpha|}{2m}} \| D^{\alpha} u \|_{L_p(\cO_T^+)}
\le N \| f\|_{L_p(\cO_T^+)},	
\end{equation*}
provided that $\lambda \ge \lambda_0$,
where $N$ and $\lambda_0 \ge 0$
depend only on $d$, $m$, $n$, $p$, $\delta$, $K$ and $R_0$.

\noindent
(ii)
For any  $\lambda > \lambda_0$, there exists a unique $u \in W_p^{1,2m}(\cO_T^+)$ satisfying \eqref{eq081902d}.
\end{theorem}

\begin{remark}
                                \label{rem2.01}
By using a scaling argument, it is easy to see that we can choose $\lambda_0$ to be zero in the theorems above provided that $\cL$ or $L$ has no lower-order terms and the leading coefficients depend only on $t$.
\end{remark}

\begin{remark}
								\label{remark01}
In the above we presented the results only for parabolic systems.
From those results we obtain easily
the corresponding results for higher order elliptic systems in divergence form and non-divergence form.
The key idea is viewing solutions to elliptic systems as steady state solutions
to the corresponding parabolic systems. We refer the reader to \cite{Krylov_2005} and \cite{DK08} for details.
To show the exact form of results for elliptic systems, we state below the cases for elliptic systems on a bounded domain, Theorem \ref{Thm9} and Theorem \ref{Thm10}.
\end{remark}

Next we consider the solvability of systems in domains with the homogeneous Dirichlet boundary condition.
For divergence systems, we assume the boundary $\partial \Omega$ of the domain $\Omega$ is locally the graph of a Lipschitz continuous function with a small Lipschitz constant. More precisely, we make the following assumption containing a parameter $\rho_1 \in (0,1]$, which will be specified later.

\begin{assumption}[$\rho_1$]
                                    \label{assump2}
There is a constant $R_1\in (0,1]$ such that, for any $x_0\in \partial\Omega$ and $r\in(0,R_1]$, there exists a Lipschitz
function $\phi$: $\bR^{d-1}\to \bR$ such that
$$
\Omega\cap B_r(x_0) = \{x\in B_r(x_0)\, :\, x^1 >\phi(x')\}
$$
and
$$
\sup_{x',y'\in B_r'(x_0'),x' \neq y'}\frac {|\phi(y')-\phi(x')|}{|y'-x'|}\le \rho_1
$$
in some coordinate system.
\end{assumption}
Note that all $C^1$ domains satisfy this assumption for any $\rho_1>0$.
Below we denote $\Omega_T  = (-\infty,T) \times \Omega$, where $\Omega \subset \bR^d$.

\begin{theorem}[Divergence parabolic systems on a bounded domain]
              \label{Thm7}
Let $p \in (1,\infty)$, $T\in (-\infty,\infty]$.
Then there exist constants $\rho=\rho(d,m,n,p,\delta)$, $\rho_1=\rho_1(d,m,n,p,\delta,K,R_0)$ and $\lambda_0=\lambda_0(d,m,n,p,\delta,K,R_0,R_1) > 0$, such that under Assumption \ref{assumption20080424} ($\rho$) and Assumption \ref{assump2}  ($\rho_1$) the following is true.
For any  $f_\alpha\in L_p(\Omega_T)$, $|\alpha|\le m$, and $\lambda \ge \lambda_0$, there is a unique solution
$u \in \cH_p^m(\Omega_T)$ to
\begin{equation*}							
\left\{
  \begin{aligned}
    u_t+(-1)^m \cL u + \lambda u = \sum_{|\alpha|\le m}D^\alpha f_{\alpha} \quad & \hbox{in $\Omega_T$;} \\
    u=|D u|=...=|D^{m-1}u|=0 \quad & \hbox{on $(-\infty,T)\times \partial \Omega$,}
  \end{aligned}
\right.
\end{equation*}
and we have
\begin{equation*}							 
\sum_{|\alpha|\le m} \lambda^{1-\frac{|\alpha|}{2m}} \| D^{\alpha} u \|_{L_p(\Omega_T)}
\le N \sum_{|\alpha| \le m} \lambda^{\frac{|\alpha|}{2m}} \| f_{\alpha} \|_{L_p(\Omega_T)},	
\end{equation*}
where $N$ depends only on $d$, $m$, $n$, $p$, $\delta$, $K$, $R_0$ and $R_1$. 
\end{theorem}

\begin{theorem}[Non-divergence parabolic systems on a bounded domain]
              \label{Thm8}
Let $p \in (1,\infty)$, $T\in (-\infty,\infty]$ and $\Omega$ be a $C^{2m-1,1}$ domain with the $C^{2m-1,1}$ norm bounded by $K$. Then there exist constants $\rho=\rho(d,m,n,p,\delta)$ and $\lambda_0=\lambda_0(d,m,n,p,\delta,K,R_0)>0$, such that under Assumption \ref{assumption20080424} ($\rho$) the following is true.
For any  $f\in L_p(\Omega_T)$ and $\lambda \ge \lambda_0$, there is a unique solution
$u \in W_p^{1,2m}(\Omega_T)$ to
\begin{equation}							 \label{eq8.10.51}
\left\{
  \begin{aligned}
    u_t+(-1)^m L u + \lambda u = f \quad & \hbox{in $\Omega_T$;} \\
    u=|D u|=...=|D^{m-1}u|=0 \quad & \hbox{on $(-\infty,T)\times \partial \Omega$,}
  \end{aligned}
\right.
\end{equation}
and we have
\begin{equation*}							 
\| u_t\|_{L_p(\Omega_T)}+\sum_{|\alpha|\le 2m} \lambda^{1-\frac{|\alpha|}{2m}} \| D^{\alpha} u \|_{L_p(\Omega_T)}
\le N \| f\|_{L_p(\Omega_T)},
\end{equation*}
where $N$ depends only on $d$, $m$, $n$, $p$, $\delta$, $K$ and $R_0$.
\end{theorem}

As discussed in Remark \ref{remark01},
the theorems above have elliptic analogies. We state the results below for elliptic systems on a bounded domain for future references.

\begin{theorem}[Divergence elliptic systems on a bounded domain]
                  \label{Thm9}

Let $p \in (1,\infty)$.
Then there exist constants $\rho=\rho(d,m,n,p,\delta)$, $\rho_1=\rho_1(d,m,n,p,\delta,K,R_0)$ and $\lambda_0=\lambda_0(d,m,n,p,\delta,K,R_0,R_1)>0$, such that under Assumption \ref{assumption20080424} ($\rho$) and Assumption \ref{assump2}  ($\rho_1$) the following is true.
For any  $f_\alpha\in L_p(\Omega),|\alpha|\le m$ and $\lambda \ge \lambda_0$, there is a unique solution
$u \in W_p^m(\Omega)$ to
\begin{equation*}							 
\left\{
  \begin{aligned}
    \cL u + (-1)^m \lambda u = \sum_{|\alpha|\le m}D^\alpha f_{\alpha} \quad & \hbox{in $\Omega$;} \\
    u=|D u|=...=|D^{m-1}u|=0 \quad & \hbox{on $\partial \Omega$,}
  \end{aligned}
\right.
\end{equation*}
and we have
\begin{equation*}							 
\sum_{|\alpha|\le m} \lambda^{1-\frac{|\alpha|}{2m}} \| D^{\alpha} u \|_{L_p(\Omega)}
\le N \sum_{|\alpha| \le m} \lambda^{\frac{|\alpha|}{2m}} \| f_{\alpha} \|_{L_p(\Omega)},	
\end{equation*}
where $N$ depends only on $d$, $m$, $n$, $p$, $\delta$, $K$, $R_0$ and $R_1$. 
\end{theorem}

\begin{theorem}[Non-divergence elliptic systems on a bounded domain]
                  \label{Thm10}
Let $p \in (1,\infty)$ and $\Omega$ be a $C^{2m-1,1}$ domain with the $C^{2m-1,1}$ norm bounded by $K$. Then there exist constants $\rho=\rho(d,m,n,p,\delta)$ and $\lambda_0=\lambda_0(d,m,n,p,\delta,K,R_0)>0$, such that under Assumption \ref{assumption20080424} ($\rho$) the following is true.
For any  $f\in L_p(\Omega)$ and $\lambda \ge \lambda_0$, there is a unique solution
$u \in W_p^{2m}(\Omega)$ to
\begin{equation*}							 
\left\{
  \begin{aligned}
    L u + (-1)^m \lambda u = f \quad & \hbox{in $\Omega$;} \\
    u=|D u|=...=|D^{m-1}u|=0 \quad & \hbox{on $\partial \Omega$,}
  \end{aligned}
\right.
\end{equation*}
and we have
\begin{equation*}							 
\sum_{|\alpha|\le 2m} \lambda^{1-\frac{|\alpha|}{2m}} \| D^{\alpha} u \|_{L_p(\Omega)}
\le N \| f\|_{L_p(\Omega)},
\end{equation*}
where $N$ depends only on $d$, $m$, $n$, $p$, $\delta$, $K$ and $R_0$.
\end{theorem}

\part{Systems in the whole space}

This part of the paper is devoted to the proofs of the $L_p$-solvability of systems in the whole space, i.e., Theorem \ref{Thm1} and \ref{Thm2}. In Section \ref{sec3} we obtain several $L_2$-estimates for systems with coefficients depending only on $t$. By using these estimates, in Section \ref{sec_aux} we prove the mean oscillation estimates for systems with the same class of coefficients. We complete the proofs of Theorem \ref{Thm1} and \ref{Thm2} in Section \ref{sec5}.

\mysection{$L_2$-estimates for systems with simple coefficients in the whole space}
                                \label{sec3}
In this section we obtain $L_2$-estimates
of parabolic systems in divergence and non-divergence form
when the coefficient matrices are measurable functions of only the time variable
satisfying the Legendre-Hadamard ellipticity condition \eqref{eq7.9.17}.
Even though our proofs are basic,
we present them here for the sake of completeness.
In particular, we derive the $L_2$-estimate of systems in non-divergence form only using that of divergence type systems.
Throughout the section we set
$$
\cL_0 u = \sum_{|\alpha|=|\beta|=m}D^{\alpha}(A^{\alpha\beta}D^{\beta} u),
$$
where $A^{\alpha\beta} = A^{\alpha\beta}(t)$.
Since $A^{\alpha\beta}$ are independent of $x \in \bR^d$, we can write
$$
\cL_0 u = \sum_{|\alpha|=|\beta|=m}A^{\alpha\beta}D^{\alpha}D^{\beta} u.
$$
Let
$C_0^{\infty}(\overline{\cO_T})$ be the collection of infinitely differentiable functions
defined on $\overline{\cO_T}$ vanishing for large $|(t,x)|$.

\begin{theorem}							\label{thVMO01}
Let $T \in (-\infty,\infty]$.	
There exists $N = N(d,n,m,\delta)$ such that,
for any $\lambda \ge 0$,
\begin{equation}							\label{eqVMO03}
\sum_{|\alpha|\le m} \lambda^{1-\frac{|\alpha|}{2m}} \| D^{\alpha} u \|_{L_2(\cO_T)}
\le N \sum_{|\alpha| \le m} \lambda^{\frac{|\alpha|}{2m}} \| f_{\alpha} \|_{L_2(\cO_T)},	
\end{equation}
if $u \in \cH_2^m(\cO_T)$, $f_{\alpha} \in L_2(\cO_T)$, $|\alpha|\le m$, and
\begin{equation}							\label{eqVMO01}
u_t + (-1)^m \cL_0 u + \lambda u = \sum_{|\alpha|\le m} D^{\alpha} f_{\alpha}
\end{equation}
in $\cO_T$.
Furthermore, for $\lambda > 0$ and $f_{\alpha} \in L_2(\cO_T)$, $|\alpha| \le m$,
there exists a unique $u \in \cH_2^m(\cO_T)$ satisfying \eqref{eqVMO01}.
\end{theorem}

\begin{proof}
We assume $\lambda > 0$.
If $\lambda = 0$, the inequality \eqref{eqVMO03} holds trivially
or we obtain
$$
\sum_{|\alpha| = m}\| D^{\alpha} u \|_{L_2(\cO_T)}
\le N \sum_{|\alpha| = m} \| f_{\alpha} \|_{L_2(\cO_T)}
\quad
\text{if}
\,\,
f_{\alpha} = 0
\,\,
\text{for}
\,\,
|\alpha| < m
$$
using the inequality \eqref{eqVMO03} for $\lambda > 0$ and letting $\lambda \searrow 0$.

Let us assume that the inequality \eqref{eqVMO03} is proved.
Then due to the fact that
$$
u_t = -(-1)^m \sum_{|\alpha| = |\beta| =m} D^{\alpha}(A^{\alpha\beta}D^{\beta}u) - \lambda u + \sum_{|\alpha| \le m}D^{\alpha}f_{\alpha},
$$
we obtain
$\|u\|_{\cH_2^m(\cO_T)}
\le N \|\cP_{\lambda} u \|_{\bH_2^{-m}(\cO_T)}$,
where $\cP_{\lambda} u = u_t + (-1)^m \cL_0 u + \lambda u$
and $N = N(d,n,m,\delta,\lambda)$.
Then using the estimate, the method of continuity, and the unique solvability of systems with coefficients $A^{\alpha\beta} = \delta_{\alpha\beta}I_{n \times n}$
we prove the second assertion of the theorem.
Therefore, we only need to prove the inequality \eqref{eqVMO03}.
Moreover, since $\cP_{\lambda}$ is a bounded linear operator from $\cH_2^m(\cO_T)$
to $\bH_2^{-m}(\cO_T)$,
it suffices to concentrate on $u \in C_0^{\infty}(\overline{\cO_T})$.

Multiply both sides of \eqref{eqVMO01} by $u$ and integrate them on $\cO_T$.
Then by integration by parts we have
\begin{equation}
									\label{eqVMO02}
\langle u,u_t \rangle_{\cO_T} + \langle D^{\alpha}u, A^{\alpha\beta}D^{\beta}u \rangle_{\cO_T}
+ \lambda \langle u, u \rangle_{\cO_T}
= \sum_{|\alpha|\le m}(-1)^{|\alpha|} \langle D^{\alpha}u, f_{\alpha} \rangle_{\cO_T}.
\end{equation}
Note that
$$
\langle D^{\alpha}u, A^{\alpha\beta}D^{\beta}u \rangle_{\cO_T}
= \langle (\ii\xi)^{\alpha}\tilde{u}, A^{\alpha\beta} (\ii\xi)^{\beta}\tilde{u} \rangle_{\cO_T}
= \int_{\cO_T}\xi^{\alpha}\xi^{\beta}\tilde{u}^{\text{tr}}\overline{A^{\alpha\beta}\tilde u} \, d\xi \, dt.
$$
Here $\tilde u$ is the Fourier transform of $u$ in $x$.
By the ellipticity condition we get
$$
\delta \int_{\cO_T} |\xi|^{2m} |\tilde u|^2 \, d\xi dt
\le \int_{\cO_T}\Re \left(\xi^{\alpha}\xi^{\beta}\tilde{u}^{\text{tr}}\overline{A^{\alpha\beta}\tilde u}\right) \, d\xi \, dt.
$$
Also note that
$$
\int_{\bR^d} |u|^2(T,x) \, dx
= \int_{\cO_T} \frac{\partial}{\partial t} |u|^2(t,x) \, dt \, dx
= \langle u, u_t \rangle_{\cO_T} + \langle u_t, u \rangle_{\cO_T},
$$
$$
\Re \langle u,u_t \rangle_{\cO_T} = \frac{1}{2}\int_{\bR^d} |u|^2(T,x) \, dx \ge 0.
$$
Thus, if we denote the right-hand side of \eqref{eqVMO02} by $I$,
we obtain
$$
\delta \int_{\cO_T} |\xi|^{2m} |\tilde u|^2 \, d\xi dt
+ \lambda \langle u,u \rangle_{\cO_T}
\le \Re I \le
\sum_{|\alpha|\le m} | \langle D^{\alpha}u,f_{\alpha} \rangle_{\cO_T}|.
$$
Since
$$
\|D^m u\|^2_{L_2(\cO_T)} \le N\int_{\cO_T} |\xi|^{2m} |\tilde u|^2 \, d\xi dt
$$
and
$$
| \langle D^{\alpha}u,f_{\alpha} \rangle_{\cO_T}| \le \varepsilon \lambda^{\frac {m-|\alpha|} m}\|D^\alpha u\|^2_{L_2(\cO_T)} + N\varepsilon^{-1} \lambda^{-\frac {m-|\alpha|} m}\|f_\alpha\|^2_{L_2(\cO_T)}
$$
for all $\varepsilon > 0$,
the inequality \eqref{eqVMO03} follows by using the interpolation inequalities and choosing an appropriate $\varepsilon$.
\end{proof}

\begin{theorem}
								\label{th06_05}
Let $T \in (-\infty,\infty]$.
There exists $N = N(d,n,m,\delta)$ such that
$$
\| u_t \|_{L_2(\cO_T)}
+ \sum_{|\alpha|\le 2m} \lambda^{1-\frac{|\alpha|}{2m}}\|D^{\alpha}u\|_{L_2(\cO_T)}
\le N \| u_t + (-1)^m\cL_0 u + \lambda u\|_{L_2(\cO_T)}
$$
for all $\lambda \ge 0$ and $u \in W_2^{1,2m}(\cO_T)$.
Moreover, for $\lambda > 0$ and $f \in L_2(\cO_T)$, there exists a unique $u \in W_2^{1,2m}(\cO_T)$
satisfying
$$
u_t + (-1)^m \cL_0 u + \lambda u = f
$$
in $\cO_T$.
\end{theorem}

\begin{proof}
As in the proof of Theorem \ref{thVMO01},
we only prove the estimate assuming that $u \in C_0^{\infty}(\overline{\cO_T})$.
Let $f = u_t + (-1)^m\cL_0 u + \lambda u$
and write
\begin{equation}							\label{eqVMO04}
u_t + (-1)^m D^{\alpha} (A^{\alpha\beta}D^{\beta} u ) + \lambda u = f.	
\end{equation}
Then by Theorem \ref{thVMO01}
\begin{equation}							\label{eqVMO05}
\lambda \| u \|_{L_2(\cO_T)} \le N \|f\|_{L_2(\cO_T)}.	
\end{equation}
Now by differentiating both sides of \eqref{eqVMO04} $m$ times with respect to $x$ we get
$$
(D^m u)_t + (-1)^m D^{\alpha} (A^{\alpha\beta}D^{\beta} D^m u )
+ \lambda D^m u = D^m f.
$$
This with Theorem \ref{thVMO01} shows that
\begin{equation}							\label{eqVMO06}
\sum_{|\alpha|=m}\|D^{\alpha}D^m u\|_{L_2(\cO_T)}
\le N \| f \|_{L_2(\cO_T)}.	
\end{equation}
Using \eqref{eqVMO05}, \eqref{eqVMO06}, and the interpolation inequalities,
we obtain
$$
\sum_{|\alpha|\le 2m} \lambda^{1-\frac{|\alpha|}{2m}}\|D^{\alpha}u\|_{L_2(\cO_T)}
\le N \| f \|_{L_2(\cO_T)}.
$$
Finally, observe that
$$
\|u_t\|_{L_2(\cO_T)} = \|f - (-1)^m\cL_0u - \lambda u\|_{L_2} \le N \| f \|_{L_2(\cO_T)}.
$$
The theorem is proved.
\end{proof}

\mysection{Mean oscillation estimates for systems in the whole space}
\label{sec_aux}

In this section we continue working on the operator
$$
\cL_0 u = \sum_{|\alpha|=|\beta|=m}A^{\alpha\beta}D^{\alpha}D^{\beta} u,
$$
where $A^{\alpha\beta}=A^{\alpha\beta}(t)$.
The main objective of this section is to obtain mean oscillation estimates
for divergence type systems (Theorem \ref{thm4.2})
and for non-divergence type systems (Corollary \ref{cor06_03})
defined in the whole space.

\subsection{Some auxiliary results for systems in the whole space}

First we prove the following localized version of Theorem \ref{th06_05}.
\begin{lemma}
                                            \label{lem06_01}
Let $0<r<R<\infty$.
Assume $u\in W_2^{1,2m}(Q_R)$ 
and
$$
u_t + (-1)^m\cL_0 u= f
$$
in $Q_R$, where $f \in L_2(\cO_T)$.
Then there exists a constant $N=N(d,n,m,\delta)$ such that
\begin{equation}
                                          \label{eq06_03}
\|u_t\|_{L_2(Q_r)} + \|D^{2m} u\|_{L_2(Q_r)}
\leq N \|f\|_{L_2(Q_R)}
+N (R-r)^{-2m}\|u\|_{L_2(Q_R)}.
\end{equation}
Furthermore,
\begin{equation}
								\label{eq06_12}
\|u\|_{W_2^{1,2m}(Q_r)}
\le N \|f\|_{L_2(Q_R)}
+ N \|u\|_{L_2(Q_R)},
\end{equation}
where
$N = N(d,n,m,\delta,r,R)$.
\end{lemma}

\begin{proof}
Let
$$
R_0 = r,
\quad
R_j = r + (R-r)\sum_{l=1}^j 2^{-l},
\quad
j = 1, 2, \cdots.
$$
For each $j = 0, 1, \cdots$,
we take $\zeta_j \in C_0^{\infty}(\bR^{d+1})$ satisfying
$$
\zeta_j =
\left\{
\begin{aligned}
1 \quad &\text{on} \quad Q_{R_j}\\
0 \quad &\text{on} \quad \bR^{d+1} \setminus (-R_{j+1}^{2m}, R_{j+1}^{2m}) \times B_{R_{j+1}}
\end{aligned}
\right.,
$$
and
$$
|D^k \zeta_j| \le N 2^{kj} (R-r)^{-k},
\quad
|(\zeta_j)_t| \le N 2^{2 mj} (R-r)^{-2m},
$$
where $k = 0,1,\cdots,2m$.
Indeed, we can take $\zeta_j$ as follows.
Let $g(z) \in C^{\infty}(\bR)$ be
a function such that
$$
0 \le g \le 1,
\quad
g(z) = 1
\,\,\,
\text{if}
\,\,\,
z \le 0,
\quad
g(z) = 0
\,\,\,
\text{if}
\,\,\,
z \ge 1/2.
$$
Then set $\zeta_j(t,x) = \psi_j(t) \eta_j(x)$,
where
$$
\psi_j(t) = g(2^j(R-r)^{-1}(|t|^{\frac{1}{2m}} - R_j)),
$$
$$
\eta_j(x) = g(2^j(R-r)^{-1}(|x|-R_j)).
$$

Now we apply Theorem \ref{th06_05} with $\lambda = 0$ to $\zeta_j u \in W_2^{1,2m}(\cO_0)$,
so that
\begin{multline}
								\label{eq06_04}
\| (\zeta_j u)_t \|_{L_2(\cO_0)}
+ \| D^{2m} (\zeta_j u) \|_{L_2(\cO_0)}
\le N \|(\zeta_j u)_t + (-1)^m \cL_0 (\zeta_j u) \|_{L_2(\cO_0)}
\\
\le N \| f \|_{L_2(Q_R)}
+ N \| (\zeta_j)_t u \|_{L_2(\cO_0)}
+ N \sum_{k=1}^{2m} \| D^k \zeta_j D^{2m-k} u \|_{L_2(\cO_0)}.	
\end{multline}
Using the properties of $\zeta_j$ and interpolation inequalities (see, for instance,
\cite{Krylov:book:2008}), for each $1\le k < 2m$,
we have
$$
\| D^k \zeta_j D^{2m-k}u \|_{L_2(\cO_0)}
= \| D^k \zeta_j D^{2m-k}(\zeta_{j+1} u) \|_{L_2(\cO_0)}
$$
$$
\le N 2^{kj}(R-r)^{-k} \|D^{2m-k}(\zeta_{j+1} u) \|_{L_2(\cO_0)}
$$
\begin{equation}
								\label{eq06_10}
\le \varepsilon \| D^{2m} (\zeta_{j+1}u) \|_{L_2(\cO_0)}
+ N 2^{2mj}(R-r)^{-2m} \| u \|_{L_2(Q_R)}.	
\end{equation}
Furthermore, we have
\begin{equation}
								\label{eq06_11}
\|(\zeta_j)_t u\|_{L_2(\cO_0)}
+ \|u D^{2m}\zeta_j\|_{L_2(\cO_0)}
\le N 2^{2mj}(R-r)^{-2m}\| u \|_{L_2(Q_R)}.
\end{equation}

Therefore, if we set
$$
I_j = \|(\zeta_j u)_t\|_{L_2(\cO_0)}
+ \|D^{2m}(\zeta_j u)\|_{L_2(\cO_0)},
$$
from \eqref{eq06_04}, \eqref{eq06_10}, and \eqref{eq06_11}
we obtain
$$
I_j \le \varepsilon I_{j+1} + N \|f\|_{L_2(Q_R)}
+ N 2^{2mj}(R-r)^{-2m}\|u\|_{L_2(Q_R)}.
$$
Multiply both sides by $\varepsilon^j$ and make summations with respect to $j$ to get
$$
\sum_{j=0}^{\infty}\varepsilon^j I_j
\le \sum_{j=1}^{\infty}\varepsilon^j I_j
+ N \sum_{j=0}^{\infty} \varepsilon^j \|f\|_{L_2(Q_R)}
+ N (R-r)^{-2m} \sum_{j=0}^{\infty} \varepsilon^j 2^{2mj} \|u\|_{L_2(Q_R)}.
$$
Upon choosing, for example, $\varepsilon = 2^{-2m-1}$,
the summations are finite, so from the above inequality
we have
$$
\| D^{2m}(\zeta_0 u) \|_{L_2(\cO_0)}
+ \| (\zeta_0 u)_t \|_{L_2(\cO_0)}
\le
N \|f\|_{L_2(Q_R)}
+ N (R-r)^{-2m} \| u \|_{L_2(Q_R)}.
$$
This proves the inequality \eqref{eq06_03}
because the left-hand side of the above inequality
is bigger than that of \eqref{eq06_03}.
Finally, the inequality \eqref{eq06_12} follows
from \eqref{eq06_03} and the interpolation inequalities.
\end{proof}

In the sequel we denote $u \in \bW_p^{1,\infty}(Q_r(t_0,x_0))$, $1<p<\infty$, $(t_0,x_0) \in \bR^{d+1}$,
if $D^{\alpha}u$, $D^{\alpha}u_t \in L_p(Q_r(t_0,x_0))$ for all multi-index $\alpha$
including $\alpha = (0,\cdots,0)$.

\begin{corollary}
                                    \label{cor33}
Let $0<r<R<\infty$ and $u\in \bW_2^{1,\infty}(Q_R)$ 
satisfy
\begin{equation}							\label{eq0609_004}
u_t + (-1)^m\cL_0 u=0
\end{equation}
in $Q_R$.
Then for any multi-index $\gamma$, we have
$$
\|D^\gamma u\|_{L_2(Q_r)}
+ \|D^\gamma u_t\|_{L_2(Q_r)} \le N\|u\|_{L_2(Q_R)},	
$$
where $N=N(d,n,m,\delta,r,R, \gamma)$.
\end{corollary}

\begin{proof}
Note that
$$
D^\gamma u_t = -(-1)^m \cL_0 D^{\gamma} u
$$
in $Q_R$.
Hence it is enough to prove
\begin{equation}					
								\label{eq06_05}
\|D^\gamma u\|_{L_2(Q_r)} \le N\|u\|_{L_2(Q_R)}.
\end{equation}
Since $u \in W_2^{1,2m}(Q_R)$,
this inequality follows from \eqref{eq06_12}
if $|\gamma| \le 2m$,
so assume that $|\gamma| > 2m$ and
$$
D^{\gamma} u = D^{2m}D^{\vartheta}u.
$$
Note that $D^{\vartheta}u$ is in $W_2^{1,2m}(Q_R)$
and satisfies \eqref{eq0609_004}.
Thus applying \eqref{eq06_03}
to the equation \eqref{eq0609_004} with $D^{\vartheta}u$ in place of $u$
we get
$$
\| D^{\gamma} u \|_{L_2(Q_r)}
\le N \| D^{\vartheta} u \|_{L_2(Q_{R_0})},
$$
where $r < R_0 < R$.
We repeat this process as many times as needed to get
$$
\| D^{\gamma} u \|_{L_2(Q_r)}
\le N \| D^{\gamma_0} u \|_{L_2(Q_{R_1})},
$$
where $|\gamma_0| \le 2m$ and $r < R_1 < R$.
Then the inequality \eqref{eq06_05}
for $|\gamma| > 2m$
follows from the same inequality for $|\gamma|\le 2m$
(with $R_1$ in place of $r$).
\end{proof}

\begin{lemma}							\label{lemma0611_01}
If $u \in \bW_2^{1,\infty}(Q_4)$ 
satisfies	
\eqref{eq0609_004} in $Q_4$,
then
\begin{equation*}							
\sup_{Q_1}|D u(t,x)|
+ \sup_{Q_1}|u_t(t,x)|
\le N \| u \|_{L_2(Q_4)},	
\end{equation*}
where $N=N(d,n,m,\delta)$.
\end{lemma}

\begin{proof}
Thanks to the fact
that $u_t = -(-1)^m \cL_0 u$ in $Q_4$,
it suffices to prove
\begin{equation}
								\label{eq06_08}
\sup_{(t,x) \in Q_1}|D^{\gamma} u(t,x)|
\le N \| u \|_{L_2(Q_4)},
\end{equation}
for a multi-index $\gamma$.
By the Sobolev embedding theorem
$$
\sup_{t\in(-1,0)}|D^{\gamma}u(t,x)|^2 \le N\int_{-1}^0|D^{\gamma}u(s,x)|^2 \, ds
+ N\int_{-1}^0 |D^{\gamma}u_t(s,x)|^2 \, ds	
$$
for each $x \in B_1$,
where $D^{\gamma}u(t,x)$ is considered as a function of $t \in (-1,0)$ for each fixed $x\in B_1$.
On the other hand, again by the Sobolev embedding theorem there exists a positive number $k$
such that
$$
\sup_{x\in B_1}|D^{\gamma} u(s,x)| \le
N \| D^{\gamma} u(s,\cdot) \|_{W_2^k(B_1)}
$$
for each $s \in (-1,0)$,
where $D^{\gamma} u(s,x)$ is considered as a function of $x \in B_1$
for each fixed $s \in (-1,0)$.
We have the same inequality as above with $D^{\gamma}u_t$ in place of $D^{\gamma}u$.
Therefore, we obtain
$$
\sup_{(t,x) \in Q_1}|D^{\gamma}u(t,x)|^2
\le N \sum_{|\vartheta|\le k}\|D^{\vartheta}D^{\gamma} u\|_{L_2(Q_1)}
+ N\sum_{|\vartheta|\le k}\|D^{\vartheta} D^{\gamma} u_t\|_{L_2(Q_1)}.
$$
This together with Corollary \ref{cor33} gives the inequality \eqref{eq06_08}.
\end{proof}

\begin{lemma}							\label{lemma6.2}
Let $\lambda \ge 0$ and $u \in \bW_2^{1,\infty}(Q_4)$ 
satisfy	
\begin{equation}
 					\label{eq10.32b}
u_t + (-1)^m \cL_0 u+\lambda u=0
\end{equation}
in $Q_4$.
Then we have
\begin{multline}							\label{eqeq061002}
\sup_{Q_1}|D^{m+1} u(t,x)|
+ \sup_{Q_1}|D^m u_t(t,x)|
+ \lambda^{\frac 1 2}\sup_{Q_1} |Du(t,x)|
\\
+ \lambda^{\frac 1 2}\sup_{Q_1} |u_t(t,x)|
\le N \sum_{k=0}^m \lambda^{\frac{1}{2}-\frac{k}{2m}} \| D^k u \|_{L_2(Q_4)},	
\end{multline}
where $N=N(d,n,m,\delta)$.
\end{lemma}

\begin{proof}
The case $\lambda =0$
follows by Lemma \ref{lemma0611_01} applied to $D^m u$
since $D^m u$ satisfies \eqref{eq0609_004}.
For the case $\lambda > 0$,
we follow an idea by S. Agmon.
Consider
$$
\zeta(y) = \cos (\lambda^{\frac{1}{2m}} y) + \sin (\lambda^{\frac{1}{2m}} y).
$$
Note that
$$
(-1)^m D^{2m}_y \zeta(y) = \lambda \zeta(y),
\quad
\zeta(0) = 1,
\quad
|D^m \zeta(0) | = \lambda^{\frac 1 2}.
$$
Denote by $(t,z) = (t,x,y)$ a point in $\bR^{d+2}$,
where $z = (x,y) \in \bR^{d+1}$,
and set
$$
\hat{u}(t,z) = u(t,x) \zeta(y),
\quad
\hat{Q}_r = (-r^{2m}, 0) \times \{ |z| < r, z \in \bR^{d+1}\}.
$$
Since $u$ satisfies \eqref{eq10.32b}, $\hat{u}$ satisfy
$$
\hat{u}_t + (-1)^m \cL_0 \hat{u} + (-1)^m D^{2m}_y \hat{u} = 0
$$
in $\hat{Q}_4$.
Upon applying the inequality \eqref{eqeq061002} with $\lambda =0$ just proved above,
we get
\begin{multline}							\label{eq0611_100}
\sup_{\hat Q_1}|D^{m+1}_x \hat u(t,z)|
+ \sup_{\hat Q_1} |D^m_x \hat u_t(t,z)|
+ \sup_{\hat Q_1}|D^m_y D_x \hat u(t,z)|
\\
+ \sup_{\hat Q_1}|D^m_y \hat u_t(t,z)|
\le N \| D^m \hat u \|_{L_2(\hat Q_4)}.
\end{multline}
Since, for example,
$$
\sup_{(t,x) \in Q_1}\lambda^{\frac 1 2}|D_x u(t,x)|
\le \sup_{(t,z) \in \hat Q_1} |D^m_y D_x \hat u (t,z)|,
$$
the left-hand side of \eqref{eq0611_100} is bigger than
that of \eqref{eqeq061002}.
On the other hand,
$D^m \hat u$ is a linear combination of terms like
$$
\lambda^{\frac{1}{2} - \frac{k}{2m}} \cos (\lambda^{\frac{1}{2m}} y) D_x^k u(t,x),
\quad
\lambda^{\frac{1}{2} - \frac{k}{2m}} \sin (\lambda^{\frac{1}{2m}} y) D_x^k u(t,x),
\quad
k = 0, \cdots, m.
$$
Thus  we see that the right-hand side of \eqref{eq0611_100} is bounded by
that of \eqref{eqeq061002}.
The lemma is proved.
\end{proof}

Recall that we denote by $X$ a point in $\bR^{d+1} = \bR \times \bR^d$.

\begin{lemma}
 				\label{lem4.1}
Let $r\in (0,\infty)$, $\kappa\in [4,\infty)$, $\lambda\ge 0$,
and $X_0 = (t_0,x_0) \in \bR^{d+1}$.
Assume $u\in \cH_{2,\text{loc}}^m (\bR^{d+1})$ 
satisfies \eqref{eq10.32b}
in $Q_{\kappa r}(X_0)$.
Then for any $\alpha$, $|\alpha|=m$,
we have
\begin{multline}	
 			\label{eq5.111}
\left(|D^\alpha u-(D^\alpha u)_{Q_r(X_0)}|\right)_{Q_r(X_0)}
+\lambda^{\frac 1 2}\left(|u-(u)_{Q_r(X_0)}|\right)_{Q_r(X_0)}
\\
\le N\kappa^{-1} \sum_{k=0}^m\lambda^{\frac 1 2-\frac k {2m}}(|D^k u|^2)_{Q_{\kappa r}(X_0)}^{\frac 1 2},
\end{multline}
where $N=N(d,n,m,\delta)>0$.
\end{lemma}

\begin{proof}
Let us prove the inequality \eqref{eq5.111}
when $X_0 = (0,0)$.	
This with a translation of the coordinates proves the inequality for general $X_0 \in \bR^{d+1}$.

Since the standard mollification of $u$ with respect to $x$
satisfies \eqref{eq10.32b} in a little bit smaller cylinder than $Q_{\kappa r}$,
we assume that $D^{\alpha} u \in L_2(Q_{\kappa r})$ for all multi-index $\alpha$.
Furthermore, \eqref{eq10.32b} implies that $D^{\gamma} u_t \in L_2(Q_{\kappa r})$
if $D^{\alpha}u \in L_2(Q_{\kappa r})$ for all $\alpha$.
Therefore, without loss of generality we assume that $u \in \bW_2^{1,\infty}(Q_{\kappa r})$.

Due to a scaling argument (for instance, see the proof of Lemma \ref{lem06_03}),
it suffices to deal with the case $r=4/\kappa$.
Observe that, for example,
$$
\left(|D^\alpha u-(D^\alpha u)_{Q_r}|\right)_{Q_r}
\le  N r \sup_{Q_1} |D^{\alpha+1} u(t,x)|
+ N r \sup_{Q_1} |D^{\alpha}u_t(t,x)|.
$$
By Lemma \ref{lemma6.2},
the right-hand side of the above inequality is bounded by
that of \eqref{eq5.111} (recall $r = 4 \kappa^{-1}$).
The lemma is proved.
\end{proof}

\subsection{Mean oscillation estimates for systems in the whole space}

In the next theorem, we prove a mean oscillation estimate for divergence form systems with simple coefficients in the whole space.

\begin{theorem}
 				\label{thm4.2}
Let $r\in (0,\infty)$, $\kappa\in [8,\infty)$, $\lambda> 0$,
$X_0 = (t_0,x_0) \in \bR^{d+1}$, and
$f_\alpha \in L_{2,\text{loc}}(\bR^{d+1})$, $|\alpha|\le m$.
Assume that $u\in \cH_{2,\text{loc}}^m (\bR^{d+1})$ 
satisfies
\begin{equation*}
u_t + (-1)^m\cL_0 u+\lambda u=\sum_{|\alpha|\le m}D^\alpha f_\alpha
\end{equation*}
in $Q_{\kappa r}(X_0)$.
Then for any $\alpha$, $|\alpha|=m$,
we have
\begin{multline}
 			\label{eq5.11}
\left(|D^\alpha u-(D^\alpha u)_{Q_r(X_0)}|\right)_{Q_r(X_0)}
+\lambda^{\frac 1 2}\left(|u-(u)_{Q_r(X_0)}|\right)_{Q_r(X_0)}
\\
\le N\kappa^{-1}\sum_{k=0}^m\lambda^{\frac 1 2-\frac k {2m}}(|D^k u|^2)_{Q_{\kappa r}(X_0)}^{\frac 12}+N\kappa^{m+\frac d 2}\sum_{|\alpha|\le m}\lambda^{\frac {|\alpha|} {2m}-\frac 1 2}(|f_\alpha|^2)_{Q_{\kappa r}(X_0)}^{\frac 12},	
\end{multline}
where $N=N(d,n,m,\delta)>0$.
\end{theorem}

\begin{proof}
We take, for the sake of simplicity,
$X_0 = (0,0)$.
As mentioned earlier, a translation gives the result for general $X_0$.

Take an infinitely differentiable function $\zeta$ defined on $\bR^{d+1}$
such that
$$
\zeta = 1 \quad \text{on} \quad Q_{\kappa r/2},
\quad
\zeta = 0 \quad \text{outside}
\quad (-(\kappa r)^{2m}, (\kappa r)^{2m}) \times B_{\kappa r}.
$$
By Theorem \ref{thVMO01}, for $\lambda > 0$, there exists a unique solution $w \in \cH_2^m(\cO_{\infty})$
to the equation
\begin{equation}
								\label{eq06_25}
w_t+(-1)^m\cL_0 w + \lambda w = \sum_{|\alpha|\le m}D^\alpha( \zeta f_\alpha )	
\end{equation}
in $\cO_{\infty} = \bR^{d+1}$.
Let $v := u - w$.
Then the function $v \in \cH_{2,\text{loc}}^m(\bR^{d+1})$
satisfies
$$
v_t+(-1)^m\cL_0v+\lambda v  = 0\quad \text{in} \quad Q_{\kappa r/2}.
$$
By Lemma \ref{lem4.1} (note that $\kappa /2 \ge 4$) applied to $v$, we have
\begin{equation}	
 			\label{eq1003}
\left(|D^\alpha v-(D^\alpha v)_{Q_r}|\right)_{Q_r}
+\lambda^{\frac 1 2}\left(|v-(v)_{Q_r}|\right)_{Q_r}
\le N\kappa^{-1}\sum_{k=0}^m\lambda^{\frac 1 2-\frac k {2m}}(|D^k v|^2)_{Q_{\kappa r/2}}^{\frac 1 2}.
\end{equation}

Next we estimate $w$,
which is the unique solution to the equation \eqref{eq06_25} considered on $\cO_0$. By Theorem \ref{thVMO01}, we have
$$
\sum_{|\alpha|\le m}\lambda^{1-\frac {|\alpha|} {2m}} \|D^\alpha w \|_{L_2(\cO_0)}
\le N \sum_{|\alpha|\le m}\lambda^{\frac {|\alpha|} {2m}} \| \zeta f_\alpha \|_{L_2(\cO_0)}.
$$
In particular,
\begin{equation}							\label{eq1001}
\left(|D^m w|^2\right)_{Q_r}^{\frac 1 2}
+ \lambda^{\frac 1 2}  \left(|w|^2\right)_{Q_{r}}^{\frac 1 2}
\le N \kappa^{m+\frac d 2} \sum_{|\alpha|\le m}\lambda^{\frac {|\alpha|} {2m}-\frac 1 2}(|f_\alpha|^2)_{Q_{\kappa r}}^{\frac 1 2},
\end{equation}
\begin{equation}							\label{eq1004}
\sum_{k=0}^m\lambda^{\frac 1 2-\frac k {2m}}\left(|D^k w|^2\right)_{Q_{\kappa r}}^{\frac 1 2}
\le N \sum_{|\alpha|\le m}\lambda^{\frac {|\alpha|} {2m}-\frac 1 2}(|f_\alpha|^2)_{Q_{\kappa r}}^{\frac 1 2}.
\end{equation}

Now we are ready to prove \eqref{eq5.11}.
Since
$$
(|D^{\alpha}u - (D^{\alpha}u)_{Q_r}|)_{Q_r}
\le 2 (|D^{\alpha}u - c|)_{Q_r}
$$
for any constant $c$, by taking $c= (D^{\alpha}v)_{Q_r}$
and repeating the same argument for the second term,
we bound the left-hand side of \eqref{eq5.11} by
a constant times
$$
(|D^{\alpha}u - (D^{\alpha}v)_{Q_r}|)_{Q_r}
+ \lambda^{\frac 1 2}(u - (v)_{Q_r})_{Q_r},
$$
which is, due to the fact that $u = w + v$,
controlled by
$$
\left(|D^\alpha v-(D^\alpha v)_{Q_r}|\right)_{Q_r}
+\lambda^{\frac 12}\left(|v-(v)_{Q_r}|\right)_{Q_r}
+\left(|D^m w|^2\right)_{Q_r}^{\frac 12}
+ \lambda^{\frac 1 2}  \left(|w|^2\right)_{Q_{r}}^{\frac 1 2}.
$$
Using \eqref{eq1003}
and \eqref{eq1001},
we see that
the above is less than
$$
N\kappa^{-1}\sum_{k=0}^m\lambda^{\frac 1 2-\frac k {2m}}(|D^k v|^2)_{Q_{\kappa r/2}}^{\frac 1 2}+N \kappa^{m+\frac d 2} \sum_{|\alpha|\le m}\lambda^{\frac {|\alpha|} {2m}-\frac 1 2}(|f_\alpha|^2)_{Q_{\kappa r}}^{\frac 1 2}.
$$
Finally, we use the fact that $v = u-w$ and
\eqref{eq1004} to prove that
the terms above are not greater than
the right-hand side of \eqref{eq5.11}.
\end{proof}

Next we consider the corresponding mean oscillation estimate for non-divergence type systems in the whole space.

\begin{corollary}
 				\label{cor06_03}
Let $r\in (0,\infty)$, $\kappa\in [8,\infty)$, $\lambda> 0$,
$X_0 \in \bR \times \bR^d$, and
$f \in L_{2,\text{loc}}(\bR^{d+1})$, $|\alpha|\le m$.
Assume that $u\in W_{2,\text{loc}}^{1,2m}(\bR^{d+1})$ 
satisfies
\begin{equation*}
u_t + (-1)^m\cL_0 u+\lambda u=f
\end{equation*}
in $Q_{\kappa r}(X_0)$.
Then for any $\alpha$, $|\alpha|=2m$,
we have
$$
\left(|D^\alpha u-(D^\alpha u)_{Q_r(X_0)}|\right)_{Q_r(X_0)}
+\lambda\left(|u-(u)_{Q_r(X_0)}|\right)_{Q_r(X_0)}
$$
$$
\le N\kappa^{-1}\sum_{k=0}^{2m}\lambda^{1-\frac k {2m}}(|D^k u|^2)_{Q_{\kappa r}(X_0)}^{\frac 1 2}
+N\kappa^{m+\frac d 2}(|f|^2)_{Q_{\kappa r}(X_0)}^{\frac 1 2},	
$$
where $N=N(d,n,m,\delta)>0$.
\end{corollary}

\begin{proof}
Again let $X_0 = (0,0)$ for simplicity.	
By Theorem \ref{thm4.2}, it follows that (after multiplying both sides by $\lambda^{\frac 1 2}$)
\begin{equation}							\label{eq061010}
\lambda (|u-(u)_{Q_r}|)_{Q_r}
\le N \kappa^{-1} \sum_{k=0}^m \lambda^{1-\frac{k}{2m}} (|D^k u|^2)_{Q_{\kappa r}}^{\frac 1 2}
+ N \kappa^{m+\frac d 2} (|f|^2)_{Q_{\kappa r}}^{\frac 1 2}.	
\end{equation}
Differentiate $m$ times both sides of the system with respect to $x$
to get
$$
D^m u_t + (-1)^m \cL_0 D^m u + \lambda D^m u = D^m f.
$$
By Theorem \ref{thm4.2} applied to $D^m u$ in place of $u$,
$$
(|D^{\gamma}D^mu - (D^{\gamma}D^mu)_{Q_r}|)_{Q_r}
\le N \kappa^{-1}\sum_{k=0}^m \lambda^{\frac{1}{2}-\frac{k}{2m}}(|D^kD^m u|^2)_{Q_{\kappa r}}^{\frac 1 2}
$$
$$
+ N \kappa^{m+\frac d 2} (|f|^2)_{Q_{\kappa r}}^{\frac 1 2},
$$
where $|\gamma|=m$.
This combined with \eqref{eq061010} gives the inequality in the corollary.
\end{proof}

\mysection{$L_p$-estimates for systems in the whole space}
							\label{sec5}

In this section, we  use the mean oscillation estimates obtained in the previous section
to prove  Theorem \ref{Thm1} and \ref{Thm2}.

Let $\cQ=\set{Q_r(t,x): (t,x) \in \bR^{d+1}, r \in (0, \infty)}$.
For a function $g$ defined on $\bR^{d+1}$,
we denote its (parabolic) maximal and sharp function, respectively, by
\begin{align*}
\cM g (t,x) &= \sup_{Q\in \cQ: (t,x) \in Q}
\dashint_{Q} | g(s,y) | \, dy \, ds,\\
g^{\#}(t,x) &= \sup_{Q\in \cQ:(t,x) \in Q}
\dashint_{Q} | g(s,y) - (g)_Q | \, dy \, ds.
\end{align*}

Then
$$
\| g \|_{L_p} \le N \| g^{\#} \|_{L_p},
\quad
\| \cM g \|_{L_p} \le N \| g\|_{L_p},
$$
if $g \in L_p$, where $1 < p < \infty$ and $N = N(d,p)$.
As is well known, the first inequality above is due to the Fefferman-Stein theorem on sharp functions and the second one is the Hardy-Littlewood maximal function theorem.

We use the idea of freezing the coefficients to obtain

\begin{lemma}
                                        \label{lem10.48}
Let $\cL$ be the operator in Theorem \ref{Thm1}.
Suppose the lower-order coefficients of $\cL$ are all zero.
Let $\mu,\nu \in (1,\infty)$, $1/\mu + 1/\nu = 1$,
and $\lambda,R \in (0, \infty)$.
Assume $u \in C_0^{\infty}(\bR^{d+1})$ vanishing outside $Q_R$
and
$$
u_t + (-1)^m\cL u+\lambda u=\sum_{|\alpha|\le m}D^\alpha f_\alpha,
$$ where $f_\alpha \in L_{2,\text{loc}}(\bR^{d+1})$.
Then there exists a constant $N = N(d,m,n,\delta,\mu)$ such that
for any $\alpha$, $|\alpha|=m$, $r \in (0,\infty)$, $\kappa \ge 8$,
and $X_0 \in \bR^{d+1}$,
we have
$$
\left(|D^\alpha u-(D^\alpha u)_{Q_r(X_0)}|\right)_{Q_r(X_0)}
+\lambda^{\frac 1 2}\left(|u-(u)_{Q_r(X_0)}|\right)_{Q_r(X_0)}
$$
$$
\le N\kappa^{-1}\sum_{k=0}^{m}\lambda^{\frac 1 2-\frac k {2m}}(|D^k u|^2)_{Q_{\kappa r}(X_0)}^{\frac 1 2}
$$
$$
+N\kappa^{m+\frac d 2}\left(\sum_{|\alpha|\le m}\lambda^{\frac {|\alpha|} {2m}-\frac 1 2}(|f_\alpha|^2)_{Q_{\kappa r}(X_0)}^{\frac 1 2}+(A_R^\#)^{\frac 1 {2\nu}} (|D^m u|^{2\mu})_{Q_{\kappa r}(X_0)}^{\frac 1 {2\mu}}\right).
$$
\end{lemma}
\begin{proof}
Let $\kappa \ge 8$ and $r \in (0, \infty)$.
Fix a $y\in \bR^d$ and set $\cL_y u =A^{\alpha\beta}(t,y) D^{\alpha}D^{\beta} u(t,x)$.
Then we have
$$
u_t + (-1)^m \cL_y u+\lambda u=\sum_{|\alpha|\le m}D^\alpha \tilde f_\alpha,
$$
where
$$
\tilde f_\alpha=f_\alpha+(-1)^m\sum_{|\beta|=m}(A^{\alpha\beta}(t,y)- A^{\alpha\beta}(t,x))D^\beta u.
$$
It follows from Theorem \ref{thm4.2} that
\begin{multline}                            \label{05162007_01}
\left(|D^\alpha u-(D^\alpha u)_{Q_r(X_0)}|\right)_{Q_r(X_0)}
+\lambda^{\frac 1 2}\left(|u-(u)_{Q_r(X_0)}|\right)_{Q_r(X_0)}\\
\le N\kappa^{-1}\sum_{k=0}^{m}\lambda^{\frac 1 2-\frac k {2m}}(|D^k u|^2)_{Q_{\kappa r}(X_0)}^{\frac 1 2}
+N\kappa^{m+\frac d 2}\sum_{|\alpha|\le m}\lambda^{\frac {|\alpha|} {2m}-\frac 1 2}(|\tilde f_\alpha|^2)_{Q_{\kappa r}(X_0)}^{\frac 1 2}.
\end{multline}
Note that
\begin{equation} \label{esti_chi}
\int_{Q_{\kappa r}(X_0)}
|\tilde f_\alpha|^2 \, dx \, dt
\le N \int_{Q_{\kappa r}(X_0)}
|f_\alpha|^2 \, dx \, dt + N I_y,
\end{equation}
where, for $|\alpha|=m$,
$$
I_y = \int_{Q_{\kappa r}(X_0)}
|(A^{\alpha\beta}(t,y)-A^{\alpha\beta}(t,x)) D^{\beta} u|^2\, dx \, dt.
$$
Denote $B$ to be $B_{\kappa r}(x_0)$ if $\kappa r<R$, or to be $B_R$ otherwise; denote $Q$ to be $Q_{\kappa r}(t_0,x_0)$ if $\kappa r<R$, or to be $Q_R$ otherwise.
Now we take average of $I_y$ with respect to $y$ in $B$.
Since $u = 0$ outside $Q_R$, by the H\"{o}lder inequality
we get
$$
\dashint_B I_y\,dy =\dashint_B\int_{Q_{\kappa r}(X_0) \cap Q_R}
|(A^{\alpha\beta}(t,y)-A^{\alpha\beta}(t,x)) D^{\beta}u|^2 \, dx \, dt\,dy
$$
$$
\leq \dashint_B\left(\int_{Q}
| A^{\alpha\beta}(t,y) - A^{\alpha\beta}(t,x)|^{2 \nu}\right)^{\frac 1 {\nu}}\,dy
\left(\int_{Q_{\kappa r}(X_0) \cap Q_R} |D^m u|^{2 \mu}\right)^{\frac 1 {\mu}},
$$
where, by the boundedness of $A^{\alpha\beta}$
as well as the definitions of $\text{osc}_x$ and $A_R^{\#}$,
the integral over $B$ in the last term above is bounded by a constant times
$$
\dashint_B\left(\int_{Q}
| A^{\alpha\beta}(t,y) - A^{\alpha\beta}(t,x)|\right)^{\frac 1 {\nu}} dy
\le \left(\dashint_B\int_{Q}
| A^{\alpha\beta}(t,y) - A^{\alpha\beta}(t,x)|\,dx\,dt\,dy\right)^{\frac 1 {\nu}}
$$
$$
\le N \left( |Q| \text{osc}_x(A^{\alpha\beta},Q)\right)^{\frac1\nu}
\le N \left(R^{2m+d} A_R^{\#}\right)^{\frac1\nu}.
$$
This together with \eqref{05162007_01} and \eqref{esti_chi} completes the proof of the lemma.
\end{proof}

\begin{proof}[Proof of Theorem \ref{Thm1}]
Due to the method of continuity, it suffices to obtain an apriori estimate. By moving all the lower-order terms to the right-hand side and taking a sufficient large $\lambda$, we may assume that all the lower-order coefficients are zero.

{\it Case 1: $p\in (2,\infty)$.} First we suppose $T=\infty$ and $u\in C_0^\infty(Q_{R_0})$. Choose a $\mu>1$ such that $2\mu<p$. Under these assumptions, from Lemma \ref{lem10.48} we easily deduce
$$
(D^\alpha u)^\#(X_0)+\lambda^{\frac 1 2} u^\#(X_0)\le N\kappa^{-1} \sum_{k=0}^{m}\lambda^{\frac 1 2-\frac k {2m}}(\cM (D^k u)^2(X_0))^{\frac 1 2}
$$
$$
+N\kappa^{m+\frac d 2}\left(\sum_{|\alpha|\le m}\lambda^{\frac {|\alpha|} {2m}-\frac 1 2}(\cM f_\alpha^2(X_0))^{\frac 1 2}+\rho^{\frac 1 {2\nu}} (\cM (D^m u)^{2\mu}(X_0))^{\frac 1 {2\mu}}\right),
$$
for any $\alpha$, $|\alpha|=m$, $r \in (0,\infty)$, $\kappa \ge 8$,
and $X_0 \in \bR^{d+1}$. This together with the interpolation inequality, the Fefferman-Stein theorem and the Hardy-Littlewood maximal function theorem yields
$$
\sum_{k=0}^{m}\lambda^{\frac 1 2-\frac k {2m}}\|D^k u\|_{L_p}
\le N\|D^\alpha u\|_{L_p}+N\lambda^{\frac 1 2} \|u\|_{L_p}
$$
\begin{equation}
                                \label{eq11.50}
\le N(\kappa^{-\frac 1 2}+\kappa^{m+\frac d 2}\rho^{\frac 1 {2\nu}}) \sum_{k=0}^{m}\lambda^{\frac 1 2-\frac k {2m}}\|D^k u\|_{L_p}+N\kappa^{m+\frac d 2}\sum_{|\alpha|\le m}\lambda^{\frac {|\alpha|} {2m}-\frac 1 2}\|f_\alpha\|.
\end{equation}
Now we can choose $\kappa$ sufficiently large and $\rho$ sufficiently small in \eqref{eq11.50} to get the desired estimate. A standard partition of the unity enables us to remove the restriction that $u\in C_0^\infty(Q_{R_0})$.
The extension to the case $T\in (-\infty,+\infty]$ is by now standard; see, for instance, \cite{Krylov_2005}.
We omit the details.

{Case 2: $p\in (1,2)$.} Since the system is in divergence form, this case follows immediately from the previous case by using the duality argument.

Finally the case $p=2$ is obtained by an interpolation argument.
\end{proof}

In a similar way, from Corollary \ref{cor06_03} we get the following counterpart of Lemma \ref{lem10.48} for non-divergence systems.

\begin{lemma}
                                        \label{lem12.52}
Let $L$ be the operator in Theorem \ref{Thm2}.
Suppose the lower-order coefficients of $L$ are all zero.
Let $\mu,\nu \in (1,\infty)$, $1/\mu + 1/\nu = 1$,
and $\lambda,R \in (0, \infty)$.
Assume $u \in C_0^{\infty}(\bR^{d+1})$ vanishing outside $Q_R$
and
$$
u_t + (-1)^m L u+\lambda u=f,
$$ where $f \in L_{2,\text{loc}}(\bR^{d+1})$.
Then there exists a constant $N = N(d,m,n,\delta,\mu)$ such that
for any $\alpha$, $|\alpha|=2m$, $r \in (0,\infty)$, $\kappa \ge 8$,
and $X_0 \in \bR^{d+1}$,
we have
$$
\left(|D^\alpha u-(D^\alpha u)_{Q_r(X_0)}|\right)_{Q_r(X_0)}
+\lambda\left(|u-(u)_{Q_r(X_0)}|\right)_{Q_r(X_0)}
$$
$$
\le N\kappa^{-1}\sum_{k=0}^{2m}\lambda^{1-\frac k {2m}}(|D^k u|^2)_{Q_{\kappa r}(X_0)}^{\frac 1 2}
$$
$$
+N\kappa^{m+\frac d 2}\left((|f|^2)_{Q_{\kappa r}(X_0)}^{\frac 1 2}+(A_R^\#)^{\frac 1 {2\nu}} (|D^{2m} u|^{2\mu})_{Q_{\kappa r}(X_0)}^{\frac 1 {2\mu}}\right).
$$
\end{lemma}

\begin{proof}[Proof of Theorem \ref{Thm2}]
As in the proof of Theorem \ref{Thm1}, it suffices to prove the apriori estimate for $T=\infty$.

{\it Case 1: $p\in (2,\infty)$.} We only need to consider the case when  $u\in C_0^\infty(Q_{R_0})$, since the general case follows from a partition of the unity. The proof of this case is almost the same as that of Theorem \ref{Thm1}, by using Lemma \ref{lem12.52} instead of Lemma \ref{lem10.48}. So we omit it.

{\it Case 2: $p\in (1,2]$.} Note that here we cannot use the duality argument directly. From Case 1 and Remark \ref{rem2.01}, we already have the $W^{1,2m}_q$ solvability of
$$
u_t+(-1)^m \cL_0 u+\lambda u=f
$$
in the whole space for any $q\in (2,\infty)$ and $\lambda>0$.
For this system, since $A^{\alpha\beta}$ are
measurable function of time only we can make use of the duality argument, which yields the solvability of the same equation for any $q\in (1,2)$. Fix a $q=(1+p)/2$. Now we can repeat the arguments in the previous section with $q$ in place of $2$, and get the estimate in Lemma \ref{lem12.52} with $q$ in place of $2$.
Finally, following the proof of Case 1 completes the proof of this case.
\end{proof}

\part{Systems on a half space or a bounded domain}

This is the most novel part of the paper. The objective of this part is to establish the $L_p$-solvability of parabolic systems on a half space or on a domain.

In the next section, we prove the $L_2$-estimates for systems with coefficients measurable in $t$ on a half space. Relying on these $L_2$-estimates, in Section \ref{sec7} we are able to derive mean oscillation estimates of some partial derivatives of solutions to systems on a half space. These estimates alone are not sufficient for our purpose. So in Section \ref{sec8} we consider a certain system with special coefficients. Combining the results in Section \ref{sec7} and \ref{sec8} together enables us to prove the $L_p$-solvability on a half space (Theorem \ref{Thm5}, \ref{Thm6}). Section \ref{sec10} is devoted to the proofs of the bounded domain cases (Theorem \ref{Thm7}, \ref{Thm8}). Finally we give several remarks about other ellipticity conditions.

\mysection{$L_2$-estimates for systems with simple coefficients on a half space}  \label{sec6}

In this section, we prove the $L_2$-estimate for systems on a half space. We again consider
$$
\cL_0 u = \sum_{|\alpha|=|\beta|=m}D^{\alpha}(A^{\alpha\beta}D^{\beta} u)
= \sum_{|\alpha|=|\beta|=m}A^{\alpha\beta}D^{\alpha}D^{\beta} u,
$$
where $A^{\alpha\beta} = A^{\alpha\beta}(t)$.
Recall that $\cO_T^+ = (-\infty, T) \times \bR^d_+$. In the divergence case (Theorem \ref{thVMO02}), the proof is rather standard. However, in the case of non-divergence systems (Theorem \ref{th06_01}), the proof is much more involved. To the best of our knowledge, Theorem \ref{th06_01} is new for higher order parabolic equations and systems with measurable coefficients depending only on $t$.

\subsection{Divergence case}

Throughout the paper, we use the notation $D_{x'}^m u$ to indicate one of $D^{\alpha}u$, where $\alpha = (\alpha_1, \cdots, \alpha_d)$, $\alpha_1 = 0$, and $|\alpha|=m$.
Sometimes, depending on the context, $D_{x'}^m u$ means the whole collection
of $D^{\alpha}u$, $|\alpha|=m$, $\alpha_1=0$.
Similar to $C_0^{\infty}(\overline{\cO_T})$,
we denote by $C_0^{\infty}(\overline{\cO^+_T})$
the collection of infinitely differentiable functions defined on $\overline{\cO^+_T}$ vanishing for large $|(t,x)| \in \overline{\cO^+_T}$.

\begin{theorem}							\label{thVMO02}
Let $T \in (-\infty,\infty]$ and $f_{\alpha} \in L_2(\cO^+_T)$.	
There exists a constant $N = N(d,n,m,\delta)$ such that
\begin{equation}							\label{eqVMO5-1}
\sum_{|\alpha|\le m} \lambda^{1-\frac{|\alpha|}{2m}} \| D^{\alpha} u \|_{L_2(\cO^+_T)}
\le N \sum_{|\alpha| \le m} \lambda^{\frac{|\alpha|}{2m}} \| f_{\alpha} \|_{L_2(\cO^+_T)}	
\end{equation}
for any $\lambda \ge 0$
and $u \in \cH_2^m(\cO^+_T)$ satisfying
\begin{equation}							\label{eqVMO5-2}
u(t,0,x') = \cdots = D^{m-1}_1u(t,0,x') = 0	
\end{equation}
on $(-\infty, T) \times \bR^{d-1}$ and
\begin{equation}
								\label{eq06_26}
u_t + (-1)^m\cL_0 u + \lambda u = \sum_{|\alpha|\le m} D^{\alpha} f_{\alpha}	
\end{equation}
in $\cO^+_T$.
Furthermore, for $\lambda > 0$ and $f_{\alpha} \in L_2(\cO^+_T)$, $|\alpha| \le m$,
there exists a unique $u \in \cH_2^m(\cO^+_T)$ satisfying \eqref{eq06_26}
in $\cO^+_T$ and \eqref{eqVMO5-2} on $(-\infty,T) \times \bR^{d-1}$.
\end{theorem}

\begin{proof}
As in the proof of Theorem \ref{thVMO01}, we consider only the case $\lambda > 0$.
We follow the lines of the proof of Theorem \ref{thVMO01}.
One noteworthy fact is that, because $u \in \cH_2^m(\cO^+_T)$ satisfies \eqref{eqVMO5-2},
we have
$$
\langle D^{\alpha} u, A^{\alpha\beta} D^{\beta} u \rangle_{\cO_T^+}
= \langle D^{\alpha} u, A^{\alpha\beta} D^{\beta} u \rangle_{\cO_T},
$$
where the function $u$ on the right-hand side is viewed as an extension of $u$ to $\cO_T$ so that it is zero on $\cO_T \setminus \cO_T^+$.
Similarly,
$$
\|D^m u \|^2_{L_2(\cO^+_T)}
\le N \int_{\cO_T} |\xi|^{2m} |\tilde u|^2 \, d\xi \, dt,
$$
where $\tilde u$ is the Fourier transform of the extension.
\end{proof}

\begin{remark}
                                \label{rem5.29}
Theorem \ref{thVMO02} can be extended to systems in a cylindrical domain $\Omega_T$, where $\Omega$ is a bounded Lipschitz domain. For small $\lambda\ge 0$, we have a better estimate than \eqref{eqVMO5-1}. Indeed, from the proof above, we get
$$
\|D^m u\|_{L_2(\Omega_T)}^2\le N\sum_{|\alpha|\le m}\|f_\alpha\|_{L_2(\Omega_T)}\|D^\alpha u\|_{L_2(\Omega_T)}.
$$
By using the Poincar\'{e} inequality,
$$
\|u\|_{L_2(\Omega_T)}\le N\|Du\|_{L_2(\Omega_T)}\le N \|D^2u\|_{L_2(\Omega_T)}\le...\le N\|D^mu\|_{L_2(\Omega_T)}.
$$
Thus, we conclude
$$
\sum_{k=0}^m \|D^k u\|_{L_2(\Omega_T)}\le N\sum_{|\alpha|\le m}\|f_\alpha\|_{L_2(\Omega_T)}.
$$
Note that in this case, the solvability also holds for $\lambda=0$.
\end{remark}

\subsection{Non-divergence case}

Let us introduce some additional notation.
Let $\tau \in \bN$ and $\{c_1, \cdots, c_{2\tau}\}$ be the solution to the following system:
\begin{equation}							\label{eq0514-02}
\sum_{k=1}^{2\tau} \left(-\frac{1}{k}\right)^j c_k = 1,
\quad
j=0,\cdots,2\tau-1.	
\end{equation}
For a function $w$ defined on $\bR^d_+$, set
\begin{equation*}
\cE_{\tau} w =
\left\{
\begin{aligned}
&w(x_1,x')	\quad \text{if} \quad x_1 > 0\\
&\sum_{k=1}^{2\tau} c_k w(-\frac{1}{k}x_1,x') \quad \text{otherwise}
\end{aligned}
\right..
\end{equation*}
Note that
$\cE_{\tau} w \in C^{2\tau-1}(\bR^d)$ if $w \in C^{\infty}(\overline{\bR^d_+})$.
Indeed, by \eqref{eq0514-02}
$$
D_1^j \left(\sum_{k=1}^{2\tau} c_k w(-\frac{1}{k}x_1,x')\right)\bigg|_{x_1=0}
= \sum_{k=1}^{2\tau} \left(-\frac1k\right)^j c_k D_1^jw(0,x')
= D_1^jw(0,x')
$$
for $j = 0, \cdots, 2\tau-1$.

We remark that similar extension operators were used in \cite{Fried76} and \cite{MR1914441} in the study of elliptic systems. We will use the following interpolation estimate.

\begin{proposition}
								\label{prop07_01}
Let $1<p<\infty$ and $u \in W_p^{m}(\bR^d_+)$.
For any $\varepsilon > 0$,
there exists $N=N(d,n,m, p, \varepsilon)$ such that
\begin{equation*}
\|D_1^kD_{x'}^{m-k} u\|_{L_p(\bR^d_+)}
\le \varepsilon \|D_1^{m}u\|_{L_p(\bR^d_+)}
+ N \sum_{j=2}^d\|D^{m}_j u\|_{L_p(\bR^d_+)},	
\end{equation*}
where $k = 0, 1, \cdots, m-1$.
\end{proposition}

\begin{proof}
Without loss of generality we assume that $u \in C_0^{\infty}(\overline{\bR^d_+})$.
Let $\hat{u} = \cE_{\tau}u$.
For a sufficiently large $\tau$, the extension $\hat{u}$
is in $W_p^{m}(\bR^d)$
and satisfies
$\sum_{j=1}^d D_j^{2m} \hat w = \sum_{j=1}^d D_j^m \hat f_j$,
where
$$
\hat f_j = \left\{
             \begin{array}{ll}
               \sum_{j=1}^d D_j^{m} u, & \hbox{in $\bR^d_+$;} \\
               \sum_{j=1}^d\sum_{k=1}^{2\tau} \hat{c}_k D_j^{m}w(-\frac1k x_1,x'), & \hbox{in $\bR^d_-$.}
             \end{array}
           \right.
$$
Here $\hat{c}_k$ are appropriate constants.
Observe that
$$
\|D_1^kD_{x'}^{m-k} u\|_{L_p(\bR^d_+)}
\le \| D^{m} \hat u \|_{L_p(\bR^d)}
\le N \|\hat f\|_{L_p(\bR^d)}
\le N \sum_{j=1}^d \| D_j^{m} u \|_{L_p(\bR^d_+)},
$$
where the second inequality is due to the $L_p$-estimate of elliptic systems in the whole space (see Remark \ref{remark02}) and $N = N(d,n,m,p)$.
By replacing $u(x_1,x')$ by $u(\varepsilon_1 x_1,x')$ in the above inequality
we have
$$
\varepsilon_1^k \|D_1^k D_{x'}^{2m-k} u\|_{L_p(\bR^d_+)}
\le \varepsilon_1^{2m} N \|D_1^{2m} u\|_{L_p(\bR^d_+)}
+ N \sum_{j=2}^d \|D_j^{2m} u\|_{L_p(\bR^d_+)}.
$$
The proposition is proved.
\end{proof}

\begin{lemma}							\label{lemma0523_01}
Let $T \in (-\infty,\infty]$.	
There exists $N = N(d,n,m,\delta)$ such that
$$
\sum_{|\alpha| = m} \|D^{\alpha}D_{x'}^m u\|_{L_2(\cO^+_T)}
+ \lambda \| u \|_{L_2(\cO^+_T)}
\le N \| u_t + (-1)^m\cL_0 u + \lambda u \|_{L_2(\cO^+_T)}
$$
for all $\lambda \ge 0$
and $u \in W_2^{1,2m}(\cO^+_T)$ satisfying
\begin{equation}							\label{eq0532_01}
u(t,0,x') = \cdots = D_1^{m-1} u(t,0,x') = 0.
\end{equation}
on $(-\infty,T) \times \bR^{d-1}$.
\end{lemma}

\begin{proof}
Define
\begin{equation}							\label{eqVMO004}
f=u_t + (-1)^m D^{\alpha} (A^{\alpha\beta}D^{\beta} u ) + \lambda u
\end{equation}
in $\cO^+_T$.
Then by Theorem \ref{thVMO02}
\begin{equation*}							
\lambda \| u \|_{L_2(\cO^+_T)} \le N \|f\|_{L_2(\cO^+_T)}.	
\end{equation*}
Now differentiate with respect to $x' \in \bR^{d-1}$ both sides of \eqref{eqVMO004} $m$ times to get
$$
(D_{x'}^m u)_t + (-1)^m D^{\alpha} (A^{\alpha\beta}D^{\beta} D_{x'}^m u )
+ \lambda D_{x'}^m u = D_{x'}^m f
$$
in $\cO^+_T$.
Note that $D_{x'}^m u$ satisfies \eqref{eq0532_01}.
Thus by Theorem \ref{thVMO02} again
we have
\begin{equation*}							
\sum_{|\alpha|=m}\|D^{\alpha}D_{x'}^m u\|_{L_2(\cO^+_T)}
\le N \| f \|_{L_2(\cO^+_T)}.
\end{equation*}
The lemma is proved.
\end{proof}

\begin{lemma}							\label{lemma0523_10}
Let $T \in (-\infty, \infty]$ and $\lambda \ge 0$.
There exists $N = N(d,n,m,\delta)$ such that,
for $u \in W_2^{1,2m}(\cO^+_T)$ satisfying \eqref{eq0532_01},
\begin{equation}							\label{eq0523_02}
\| D_1^{2m} u \|_{L_2(\cO^+_T)}
\le N \sum_{j = 1}^{2m} \| D_1^{2m-j}D_{x'}^ju\|_{L_2(\cO^+_T)}
+ N \|f\|_{L_2(\cO^+_T)}
\end{equation}
provided that
\begin{equation}							\label{eq0523_03}
u_t + (-1)^m\cL_0 u + \lambda u = f
\end{equation}
in $\cO^+_T$.
\end{lemma}

\begin{proof}
By multiplying both sides of the equation \eqref{eq0523_03} from the left by $D_1^{2m}u$ we get
\begin{equation}							\label{eq0609_001}
\langle D_1^{2m}u, u_t \rangle_{\cO^+_T}
+ (-1)^m \langle D_1^{2m}u, A^{\alpha\beta}D^{\alpha}D^{\beta}u \rangle_{\cO^+_T}
+ \lambda \langle D_1^{2m}u, u \rangle_{\cO^+_T}
= \langle D_1^{2m}u, f \rangle_{\cO^+_T}.	
\end{equation}
Note that
\begin{equation}							\label{eq0609_002}
\Re (-1)^m \langle D_1^{2m}u, u_t \rangle_{\cO^+_T}  = \frac{1}{2}\int_{\bR^d_+} |D_1^m u|^2(T,x) \, dx \ge 0.	
\end{equation}
Indeed, this holds true because
$$
\int_{\bR^d_+} |D_1^m u|^2(T,x) \, dx = \int_{\cO^+_T} \frac{\partial}{\partial t} | D_1^m u |^2 \, dx \, dt
= \langle D_1^m u, D_1^m u_t \rangle_{\cO^+_T}
+ \overline{ \langle D_1^m u, D_1^m u_t \rangle_{\cO^+_T}}	
$$
and
$$
\langle D_1^{2m}u, u_t \rangle_{\cO^+_T} = (-1)^m \langle D_1^m u, D_1^m u_t \rangle_{\cO^+_T},
$$
the latter of which follows from the boundary condition \eqref{eq0532_01} and integration by parts.
Hence by taking the real parts of \eqref{eq0609_001} and using \eqref{eq0609_002}
we have
$$
\Re \langle D_1^{2m} u, A^{\hat{\alpha}\hat{\alpha}}D_1^{2m}u \rangle_{\cO^+_T}
\le -\Re \sum_{(\alpha,\beta) \ne (\hat{\alpha},\hat{\alpha})}
\langle D_1^{2m} u, A^{\alpha\beta}D^{\alpha}D^{\beta}u \rangle_{\cO^+_T}
$$
$$
- (-1)^m \lambda \Re \langle D_1^{2m}u, u \rangle_{\cO^+_T}
+ (-1)^m \Re \langle D_1^{2m}u, f \rangle_{\cO^+_T},
$$
where $\hat{\alpha} = (m,0,\cdots,0)$.
Thanks to the ellipticity condition and Young's inequality,
$$
\delta \| D_1^{2m}u \|^2_{L_2(\cO^+_T)}
\le \Re \langle D_1^{2m} u, A^{\hat{\alpha}\hat{\alpha}}D_1^{2m}u \rangle_{\cO^+_T}
\le \varepsilon \| D_1^{2m}u \|^2_{L_2(\cO^+_T)}
$$
$$
+ N(\varepsilon,\delta) \sum_{j=1}^{2m}\|D_1^{2m-j}D_{x'}^j u\|^2_{L_2(\cO^+_T)}
+ N(\varepsilon) \lambda^2 \|u\|^2_{L_2(\cO^+_T)} + N(\varepsilon)\|f\|^2_{L_2(\cO^+_T)}.
$$
Choosing a sufficiently small $\varepsilon$ and using Lemma \ref{lemma0523_01},
we prove \eqref{eq0523_02}.
\end{proof}

Now we are ready to state and prove the main theorem of the section.

\begin{theorem}							\label{th06_01}
Let $T \in (-\infty,\infty]$.
There exists $N = N(d,n,m,\delta)$ such that
$$
\| u_t \|_{L_2(\cO^+_T)}
+ \sum_{|\alpha|\le 2m} \lambda^{1-\frac{|\alpha|}{2m}}\|D^{\alpha}u\|_{L_2(\cO^+_T)}
\le N\| u_t + (-1)^m\cL_0 u + \lambda u \|_{L_2(\cO^+_T)}
$$
for all $\lambda \ge 0$
and $u \in W_2^{1,2m}(\cO^+_T)$ satisfying
\begin{equation*}							
u(t,0,x') = \cdots = D_1^{m-1} u(t,0,x') = 0
\end{equation*}
on $(-\infty,T) \times \bR^{d-1}$.
\end{theorem}

\begin{proof}
Thanks to Lemma \ref{lemma0523_01} and interpolation inequalities,
it suffices to prove
that
\begin{equation}
								\label{eq07_07}
\|u_t\|_{L_2(\cO^+_T)}
+ \|D^{2m} u\|_{L_2(\cO^+_T)}
\le N \| f \|_{L_2(\cO^+_T)},	
\end{equation}
where $f = u_t + (-1)^m \cL_0 u + \lambda u$.
Lemma \ref{lemma0523_10} and Proposition \ref{prop07_01} (with $2m$ in place of $m$)
imply that
$$
\|D^{2m} u\|_{L_2(\cO^+_T)}
\le N \|f\|_{L_2(\cO^+_T)}
+ \varepsilon \|D_1^{2m} u\|_{L_2(\cO^+_T)}
+ N \|D_{x'}^{2m} u\|_{L_2(\cO^+_T)}.
$$
This along with Lemma \ref{lemma0523_01} and a sufficiently small $\varepsilon$
proves the inequality \eqref{eq07_07}
without the $u_t$ term on the left-hand side.
To complete the proof we simply note that
$$
u_t = -(-1)^m\cL_0 u - \lambda u + f.
$$
\end{proof}

\mysection{Mean oscillation estimates of some partial derivatives of solutions to systems on a half space}
                            \label{sec7}

The aim of this section is to derive several mean oscillation estimates of highest order derivatives of solutions to systems on a half space. Contrary to the whole space case, here we are only able to estimate parts of the highest order derivatives.
More precisely, for divergence form systems, we give estimate of $D_{x'}^m u$, while for non-divergence form systems we present the estimate of $D_{x'}^{2m} u$. We emphasize that these estimates alone are not sufficient for proving Theorem \ref{Thm5} and \ref{Thm6}.

We still denote
$$
\cL_0 u = \sum_{|\alpha|=|\beta|=m} D^{\alpha} A^{\alpha\beta}(t) D^{\beta} u
= \sum_{|\alpha|=|\beta|=m} A^{\alpha\beta}(t) D^{\alpha} D^{\beta} u.
$$
Recall that
$$
Q_r^+(t,x)=Q_r(t,x) \cap \cO_{\infty}^+,
\quad
Q_r^+ = Q_r \cap \cO_{\infty}^+,
$$
$$
Q_r' = (-r^{2m},0) \times B_r',
\quad
B_r' = \{ x' \in \bR^{d-1}: |x'| < r\}.
$$

\subsection{Some auxiliary results for systems on a half space}

We first prove some auxiliary estimates in this subsection. The first two are counterparts of Lemma \ref{lem06_01} and Corollary \ref{cor33}.

\begin{lemma}
                                            \label{lem6.2}
Let $0<r<R<\infty$. Assume that
$u \in W_p^{1,2m}(Q_R^+)$ 
satisfies
\begin{equation}
								\label{eq06_01}
u(t,0,x') = \cdots = D_1^{m-1}u(t,0,x') = 0
\end{equation}
on $Q_R'$ and
$$
u_t + (-1)^m\cL_0 u=f
$$
in $Q_R^+$, where $f \in L_2(Q_R^+)$.
Then there exists a constant $N=N(d,n,m,\delta)$ such that
\begin{equation*}
\|u_t\|_{L_2(Q_r^+)}
+ \|D^{2m} u\|_{L_2(Q_r^+)}
\leq N \|f\|_{L_2(Q_R^+)}
+N(R-r)^{-2m}\|u\|_{L_2(Q_R^+)}.
\end{equation*}
Furthermore,
\begin{equation*}
\|u\|_{W_2^{1,2m}(Q_r^+)}
\le N \|f\|_{L_2(Q_R^+)}
+ N \|u\|_{L_2(Q_R^+)},
\end{equation*}
where $N = N(d,n,m,\delta, r,R)$.
\end{lemma}

\begin{proof}
By Theorem \ref{th06_01} the $L_2$-estimate of systems on a half spaces is available.
Then the proof is the same as that of Lemma \ref{lem06_01}
with some minor changes.
\end{proof}	


\begin{corollary}
								\label{cor06_01}
Let $0 < r < R < \infty$.
Assume that $u \in C_{\text{loc}}^{\infty}(\overline{\cO_{\infty}^+})$
satisfies \eqref{eq06_01} on $Q_R'$ and
\begin{equation}
								\label{eq06_09}
u_t + (-1)^m \cL_0 u = 0	
\end{equation}
in $Q_R^+$.
Then for any multi-indices $\gamma$
and $\vartheta$ such that
$$
\gamma = (\gamma_1, \gamma_2, \cdots, \gamma_d),
\quad \gamma_1 \le 2m,
\quad
\vartheta = (0, \vartheta_2, \cdots, \vartheta_d),
$$
we have
$$
\| D^{\gamma} u \|_{L_2(Q_r^+)}
+ \| D^{\vartheta} u_t \|_{L_2(Q_r^+)}
\le N \| u \|_{L_2(Q_R^+)},
$$
where $N = N(d,n,m,\delta,r,R, \gamma, \vartheta)$.
\end{corollary}

\begin{proof}
From \eqref{eq06_09} it follows that
$$
D^{\theta}u_t = -(-1)^m A^{\alpha\beta}D^{\alpha}D^{\beta}D^{\theta}u
$$
in $Q_R^+$. Each of the terms on the right-hand side is
a constant times a term of the form $D^{\gamma}u$,
where $|\gamma| = 2m + |\beta|$ and $\gamma_1 \le 2m$.
Hence we only need to prove
$$
\|D^{\gamma} u\|_{L_2(Q_r^+)}
\le N \|u\|_{L_2(Q_R^+)},	
$$
where $\gamma = (\gamma_1,\cdots,\gamma_d)$ satisfies $\gamma_1\le 2m$.
The proof of this inequality is identical to that of \eqref{eq06_05}
in Corollary \ref{cor33}
with the only difference that, in $|\gamma| > 2m$, we write
$$
D^{\gamma}u = D^{2m}D^{\vartheta}u,
\quad
\vartheta=(0, \vartheta_2,\cdots,\vartheta_d),
$$
where $D^{\vartheta}u$ satisfies \eqref{eq06_09}
in $Q_R^+$ as well as \eqref{eq06_01} on $Q_R'$,
so that we can apply Lemma \ref{lem6.2} to $D^{\vartheta}u$.
\end{proof}

Next we derive a few H\"older estimates of solutions. Throughout the paper, for a function $g$ defined on a subset $\cD$ in $\bR^{d+1}$,
we set
$$
[g]_{\cC^{\nu}(\cD)}
= \sup_{\substack{(t,x), (s,y) \in \cD \\ (t,x) \ne (s,y)}}
\frac{|g(t,x)-g(s,y)|}{|t-s|^{\frac \nu 2} + |x-y|^\nu},
$$
where $0 < \nu \le 1$.

\begin{lemma}
								\label{lem06_02}
If $u \in C_{\text{loc}}^{\infty}(\overline{\cO_{\infty}^+})$
satisfies \eqref{eq06_01} on $Q'_4$
and \eqref{eq06_09} in $Q_4^+$, then
$$
[u]_{\cC^{1/2}(Q_1^+)}
\le N \| u \|_{L_2(Q_4^+)}.
$$
\end{lemma}

\begin{proof}
Let	
$$
\Theta_r^+ = \{(t,x_1) \in (-r^{2m},0) \times (0,r)\},
\quad
B_r' = \{x' \in \bR^{d-1}: |x'| < r\}.
$$
The triangle inequality gives
$$
\sup_{\substack{(t,x),(s,y) \in Q_1^+\\ (t,x) \ne (s,y)}}
\frac{|u(t,x) - u(s,y)|}{|t-s|^{1/4}+|x-y|^{1/2}}
\le
\sup_{\substack{(t,x_1),(s,y_1) \in \Theta_1^+\\ x' \in B_1'}}
\frac{|u(t,x_1,x') - u(s,y_1,x')|}{|t-s|^{1/4}+|x_1-y_1|^{1/2}}
$$
$$
+\sup_{\substack{(s,y_1) \in \Theta_1^+\\x',y' \in B_1', x'\ne y'}}
\frac{|u(s,y_1,x') - u(s,y_1,y')|}{|x'-y'|^{1/2}}
:= I_1 + I_2.
$$

To estimate $I_1$, we view $u(t,x_1,x')$ as a function of $(t,x_1)$
for a fixed $x' \in B_1'$.
Then by the Sobolev embedding theorem
\begin{equation}
								\label{eq06_15}
\sup_{(t,x_1),(s,y_1) \in \Theta_1^+}
\frac{|u(t,x_1,x') - u(s,y_1,x')|}{|t-s|^{1/4}+|x_1-y_1|^{1/2}}
\le N \| u(\cdot,x') \|_{W_2^{1,2}(\Theta_1^+)}	
\end{equation}
for each $x' \in B_1'$.
On the other hand, there exists a positive integer $k$
such that, for each $(t,x_1) \in \Theta_1^+$,
\begin{multline}
								\label{eq06_16}
\sum_{j=0}^2\sup_{x'\in B_1'}|D_1^j u(t,x_1,x')|
+ \sup_{x'\in B_1'}|u_t(t,x_1,x')|
\\
\le N \sum_{j=0}^2\|D_1^j u(t,x_1,\cdot)\|_{W_2^k(B_1')}
+ N \|u_t(t,x_1,\cdot)\|_{W_2^k(B_1')},
\end{multline}
where $D^j_1u(t,x_1,x')$ and $u_t(t,x_1,x')$ are viewed as functions of $x' \in B_1'$.
Combining \eqref{eq06_15} and \eqref{eq06_16} proves
\begin{equation}
									\label{eq06_17}
I_1 \le N \sum_{\substack{|\gamma| \le k+2\\\gamma_1 \le 2}} \| D^{\gamma} u \|_{L_2(Q_2^+)}
+ N \sum_{\substack{|\vartheta| \le k\\\vartheta_1 = 0}}\|D^{\vartheta}u_t\|_{L_2(Q_2^+)}
\le N \|u\|_{L_2(Q_4^+)},	
\end{equation}
where the last inequality is due to Corollary \ref{cor06_01}.

For the estimate of $I_2$, we look at $u(s,y_1,x')$ as a function of $x' \in B_1'$
for each $(s,y_1) \in \Theta_1^+$.
Again by the Sobolev embedding theorem, there exists a sufficiently large integer $k$ such that
$$
\sup_{\substack{x',y' \in B_1'\\x'\ne y'}}
\frac{|u(s,y_1,x') - u(s,y_1,y')|}{|x'-y'|^{1/2}}
\le N \|u(s,y_1,\cdot)\|_{W_2^k(B_1')}.
$$
Moreover, as a function of $(s,y_1) \in \Theta_1^+$,
$D_{x'}^ju(s,y_1,x')$, $j=0,\cdots,k$, satisfy
$$
\sum_{j=0}^k\sup_{(s,y_1) \in \Theta_1^+}
|D^j_{x'}u(s,y_1,x')| \le
N \sum_{j=0}^k \|D^j_{x'}u(\cdot,x')\|_{W_p^{1,2}(\Theta_1^+)}
$$
for each $x'\in B_1'$.
From the above two inequalities, we obtain \eqref{eq06_17} with $I_2$ in place of $I_1$.
The lemma is proved.
\end{proof}

In the sequel,
for a function $g$ defined on $\cO_T^+$, $T \in (-\infty, \infty]$,
we denote by $\cE(g)$ ($=\cE g$) the even extension of $g$ defined
on $\cO_T$.

\begin{corollary}
								\label{cor06_02}
Let $\lambda \ge 0$,
$X_0 = (t_0,0,x_0')$, where $t_0 \in \bR$
and $x'_0 \in \bR^{d-1}$.
Assume that $u \in C_{\text{loc}}^{\infty}(\overline{\cO_{\infty}^+})$
satisfies \eqref{eq06_01} on $Q_4'(X_0) = (t_0-4^{2m},t_0) \times B_4'(x'_0)$ 
and
\begin{equation}							\label{eq06_02}
u_t + (-1)^m \cL_0 u + \lambda u = 0	
\end{equation}
in $Q_4^+(X_0)$.
Then there exists $N = N(d,n,m,\delta)$ such that
$$
[\cE(D^m_{x'} u)]_{\cC^{1/2}(Q_1(X_0))}
+ \lambda^{1/2}[\cE u]_{\cC^{1/2}(Q_1(X_0))}
\le N \sum_{k=0}^{m} \lambda^{\frac{1}{2}-\frac{k}{2m}}\| \cE (D^k u) \|_{L_2(Q_4(X_0))}.
$$
\end{corollary}

\begin{proof}
By using a translation in $t$ and $x'$, we assume that $X_0 = (0,0)$.
Let $\lambda =0$. In this case,
the inequality in the corollary follows from
$$
[D^m_{x'} u]_{\cC^{1/2}(Q_1^+)}
\le N \| D^m u \|_{L_2(Q_4^+)}.
$$
To prove this, we apply Lemma \ref{lem06_02} to $D^m_{x'}u$.
This is indeed possible because $D^m_{x'} u$ satisfies
\eqref{eq06_01} on $Q_4'$ 
and \eqref{eq06_09} in $Q_4^+$.
To prove the case $\lambda > 0$, we follow the steps in the proof of Lemma \ref{lemma6.2}.
\end{proof}

\begin{lemma}
								\label{lem06_03}
Let $r \in (0,\infty)$, $\kappa \in [64,\infty)$, $\lambda \ge 0$,
and $X_0 = (t_0,x_0) \in \overline{\cO_\infty^+}$.
Assume that $u \in C_{\text{loc}}^{\infty}(\overline{\cO_{\infty}^+})$
satisfies \eqref{eq06_01} on $\bR \times \bR^{d-1}$
and \eqref{eq06_02} in $Q_{\kappa r}^+(X_0)$.
Then
\begin{multline}
								\label{eq06_18}
(  |\cE (D_{x'}^m u) - (\cE (D_{x'}^m u))_{Q_r(X_0)}  |)_{Q_r(X_0)}
+ \lambda^{\frac 1 2} ( | \cE u - (\cE u)_{Q_r(X_0)}  |)_{Q_r(X_0)}
\\
\le N \kappa^{-\frac 1 2} \sum_{k=0}^m \lambda^{\frac{1}{2}-\frac{k}{2m}} (|\cE(D^k u)|^2)^{\frac 1 2}_{Q_{\kappa r}(X_0)},
\end{multline}
where $N = N(d,n,m,\delta)$.
\end{lemma}

\begin{proof}
We first prove that, using a scaling argument,
it suffices to prove the inequality \eqref{eq06_18} only for $r=16/\kappa$.
Indeed, assume that the inequality \eqref{eq06_18} holds true for $r=16/\kappa$.
For a given $r \in (0,\infty)$, let $r_0 = 16/\kappa$, $R = r/r_0$, and $w(t,x) = u(R^{2m}t,Rx)$.
It is easy to check that
$w$ satisfies \eqref{eq06_01} on $\bR \times \bR^{d-1}$ and
\begin{equation}							\label{eq2001}
w_t + (-1)^mA^{\alpha\beta}(R^{2m}t) D^{\alpha}D^{\beta}w + \lambda R^{2m} w = 0	
\end{equation}
in $Q_{\kappa r_0}^+(Y_0)$, where $Y_0 = (s_0,y_0) = (R^{-2m}t_0, R^{-1}x_0) \in \overline{\cO_\infty^+}$.
Then by the inequality \eqref{eq06_18} with $r=r_0$ applied to the system \eqref{eq2001},
we have
\begin{multline}							\label{eq0010}
\left(|\cE (D^m_{x'} w) - (\cE(D^m_{x'}w))_{Q_{r_0}(Y_0)}|\right)_{Q_{r_0}(Y_0)}
+ \lambda^{\frac 1 2}R^m\left(|\cE w-(\cE w)_{Q_{r_0}(Y_0)}|\right)_{Q_{r_0}(Y_0)}
\\
\le N\kappa^{-\frac 1 2} \sum_{k=0}^m\lambda^{\frac 1 2-\frac k {2m}}R^{m-k}(|\cE(D^k w)|^2)_{Q_{\kappa {r_0}}(Y_0)}^{\frac 1 2}.
\end{multline}
Note that, for example,
$$
(\cE(D^k w))_{Q_{r_0}(Y_0)} = R^k (\cE(D^k u))_{Q_r(X_0)}.
$$
Thus the inequality \eqref{eq0010} leads to the inequality \eqref{eq06_18}
for arbitrary $r \in (0,\infty)$.

Now we assume $r=16/\kappa$. We consider two cases.

{\em Case 1: the first coordinate of $x_0 \ge 1$.}
In this case, we see that
$Q_{\kappa r/16}^+(X_0) = Q_{\kappa r/16}(X_0)$
and $u$ satisfies the assumptions in Lemma \ref{lem4.1},
especially, $u$ satisfies \eqref{eq06_02} in $Q_{\kappa r/16}(X_0)$
and $u$ can be extended to a function in $\cH_{2,\text{loc}}^m(\bR^{d+1})$
without changing the values of $u$ on $Q_{\kappa r/16}(X_0)$.
Hence by the inequality \eqref{eq5.111} with $Q_{\kappa r/16}(X_0)$
in place of $Q_{\kappa r}(X_0)$ (note that $\kappa/16 \ge 4$),
the left-hand side of \eqref{eq06_18} is controlled by
$$
N \kappa^{-1} \sum_{k=0}^n\lambda^{\frac 1 2-\frac k {2m}}(|D^k u|^2)_{Q_{\kappa r/16}(X_0)}^{\frac 1 2},
$$
which is less than the right-hand side of \eqref{eq06_18}.

{\em Case 2: the first coordinate of $x_0$ is in $[0,1]$.}
By denoting $Y_0 =(t_0,0,x_0')$, we have
$$
Q_r(X_0)\subset Q_{2}(Y_0)\subset Q_{8}(Y_0)
\subset Q_{\kappa r}(X_0).
$$
By Corollary \ref{cor06_02}
applied to $u$ with $2$ and $8$ in place of $1$ and $4$, respectively
(this case can be seen using a scaling argument as above),
we have
$$
(|\cE (D_{x'}^m u) - (\cE (D_{x'}^m u))_{Q_r(X_0)}  |)_{Q_r(X_0)}
\le N r^{\frac 1 2} [\cE (D_{x'}^m u)]_{\cC^{1/2}(Q_2(Y_0))}
$$
$$
\le N \kappa^{-\frac 1 2} \sum_{k=0}^m\lambda^{\frac 1 2-\frac k {2m}}(|\cE(D^k u)|^2)_{Q_8(Y_0)}^{\frac 1 2}
\le N \kappa^{-\frac 1 2} \sum_{k=0}^m\lambda^{\frac 1 2-\frac k {2m}}(|\cE(D^k u)|^2)_{Q_{\kappa r}(X_0)}^{\frac 1 2}.
$$
The second term on the left-hand side of \eqref{eq06_18} are estimated similarly.
\end{proof}

\subsection{Mean oscillation estimates of $D^m_{x'}u$ for divergence type systems
on a half space}

Now we state and prove the main result of this section.

\begin{proposition}
								\label{thm07_02}
Let $r \in (0,\infty)$, $\kappa \in [128,\infty)$, $\lambda \ge 0$,
and $X_0 = (t_0,x_0) \in \overline{\cO_{\infty}^+}$.
Assume that $u \in \cH_{2,\text{loc}}^m(\cO_{\infty}^+)$ 
satisfies \eqref{eq06_01} on $\bR \times \bR^{d-1}$ and
\begin{equation}
								\label{eq07_08}
u_t + (-1)^m \cL_0 u + \lambda u = \sum_{|\alpha|\le m} D^{\alpha}f_{\alpha}	
\end{equation}
in $Q_{\kappa r}^+(X_0)$,
where $f_{\alpha} \in L_{2,\text{loc}}(\cO_{\infty}^+)$, $|\alpha| \le m$.
Then we have
\begin{multline}
								\label{eq0917}
(  |\cE (D_{x'}^m u) - (\cE (D_{x'}^m u))_{Q_r(X_0)}  |)_{Q_r(X_0)}
+ \lambda^{\frac 1 2} ( | \cE u - (\cE u)_{Q_r(X_0)}  |)_{Q_r(X_0)}
\\
\le N \kappa^{-\frac12} \sum_{k=0}^m \lambda^{\frac{1}{2}-\frac{k}{2m}} (|\cE(D^k u)|^2)^{\frac 1 2}_{Q_{\kappa r}(X_0)}
+ N \kappa^{m+\frac d2} \sum_{|\alpha|\le m} \lambda^{\frac{|\alpha|}{2m}-\frac12}
(|\cE f_{\alpha}|^2)_{Q_{\kappa r}(X_0)}^{\frac 1 2},
\end{multline}
where $N = N(d,n,m,\delta)$.
\end{proposition}

\begin{proof}
Multiplying $u$ by an infinitely smooth function as $\zeta$ below,
we see that \eqref{eq07_08} can be extended to a system defined on $\cO_{\infty}^+$
without changing the values of $u$ and $f_{\alpha}$ on, for example, $Q_{\kappa r/2}$.
Thus without loss of generality we assume that
$u \in \cH_2^m(\cO_\infty^+)$, $f_{\alpha} \in L_2(\cO_\infty^+)$,
and \eqref{eq07_08} is satisfied in $\cO_\infty^+$.
We consider only $\lambda > 0$.

Take a $\zeta \in C_0^{\infty}(Q_{\kappa r}(X_0))$ such that
$$
\zeta = 1
\,\,\,
\text{on}
\,\,\,
Q_{\kappa r/2}(X_0),
\,\,\,
\zeta = 0
\,\,\,
\text{outside}
\,\,\,
(t_0 - (\kappa r)^{2m}, t_0 + (\kappa r)^{2m}) \times B_{\kappa r}(x_0).
$$
Let $\cL_0^{(\varepsilon)} = A_{(\varepsilon)}^{\alpha\beta}D^{\alpha}D^{\beta}$,
where $A_{(\varepsilon)}^{\alpha\beta}$
are the standard mollifications with respect to $t$ of $A^{\alpha\beta}(t)$.
Also let $f^{(\varepsilon)}_{\alpha}$
be infinitely differentiable functions approaching $f_{\alpha}$ in $L_2(\cO_{\infty}^+)$
as $\varepsilon \searrow 0$.
By Theorem \ref{thVMO02}, there exists a unique solution $v^{(\varepsilon)} \in \cH_2^m(\cO_\infty^+)$, satisfying \eqref{eq06_01} on $\bR \times \bR^{d-1}$,
to the equation
\begin{equation*}
v^{(\varepsilon)}_t + (-1)^m \cL_0^{(\varepsilon)} v^{(\varepsilon)} + \lambda v^{(\varepsilon)} = \sum_{|\alpha|\le m}D^\alpha((1- \zeta) f^{(\varepsilon)}_\alpha )
\end{equation*}
in $\cO_{\infty}^+$.
Since $f_{\alpha}^{(\varepsilon)}$ and $A_{(\varepsilon)}^{\alpha\beta}$ are infinitely differentiable,
by the classical theory for higher order parabolic systems, $v^{(\varepsilon)}$ is infinitely differentiable.
Moreover, for any $\varepsilon$,
$$
v^{(\varepsilon)}_t + (-1)^m\cL^{(\varepsilon)}_0 v^{(\varepsilon)} + \lambda v^{(\varepsilon)}  = 0 \quad \text{in}\,\,Q^+_{\kappa r/2}(X_0).
$$
Thus by Lemma \ref{lem06_03} (note that $\kappa /2 \ge 64$)
\begin{multline*}
(  |\cE (D_{x'}^m v^{(\varepsilon)}) - (\cE (D_{x'}^m v^{(\varepsilon)}))_{Q_r(X_0)}  |)_{Q_r(X_0)}
+ \lambda^{\frac 1 2} ( | \cE v^{(\varepsilon)} - (\cE v^{(\varepsilon)})_{Q_r(X_0)}  |)_{Q_r(X_0)}
\\
\le N \kappa^{-\frac 1 2} \sum_{k=0}^m \lambda^{\frac{1}{2}-\frac{k}{2m}} (|\cE(D^k v^{(\varepsilon)})|^2)^{\frac 1 2}_{Q_{\kappa r}(X_0)}.
\end{multline*}

Set $w^{(\varepsilon)} = u - v^{(\varepsilon)}$.
Then $w^{(\varepsilon)} \in \cH_2^m(\cO_{\infty}^+)$ and it satisfies
\eqref{eq06_01} on $\bR \times \bR^{d-1}$
and
$$
w^{(\varepsilon)}_t + (-1)^m\cL^{(\varepsilon)}_0 w^{(\varepsilon)}+ \lambda w^{(\varepsilon)}  =
D^{\alpha}( \zeta f_{\alpha}^{(\varepsilon)} + f_{\alpha} - f_{\alpha}^{(\varepsilon)})
+ (-1)^m (\cL_0^{(\varepsilon)} - \cL_0 ) u
$$
in $\cO_\infty^+$.
Denote the right-hand side of the above equality by $D^{\alpha} g_{\alpha}^{(\varepsilon)}$.
We apply Theorem \ref{thVMO02} to the above equation
as one defined on $\cO_{t_0}^+$ so that we have
$$
\sum_{|\alpha|\le m}\lambda^{1-\frac {|\alpha|} {2m}} \|D^\alpha w^{(\varepsilon)} \|_{L_2(\cO_{t_0}^+)}
\le N \sum_{|\alpha|\le m}\lambda^{\frac {|\alpha|} {2m}} \| g_{\alpha}^{(\varepsilon)} \|_{L_2(\cO_{t_0}^+)}.
$$
In particular,
\begin{multline}							\label{eq9001}
\|D^m w^{(\varepsilon)}\|_{L_2(Q_r^+(X_0))}
+ \lambda^{\frac 1 2}  \|w^{(\varepsilon)} \|_{L_2(Q_r^+(X_0))}
\\
\le N \sum_{|\alpha|\le m}\lambda^{\frac {|\alpha|} {2m}-\frac 1 2}\|f^{(\varepsilon)}_{\alpha}\|_{L_2(Q_{\kappa r}^+(X_0))}
+ I^{(\varepsilon)},
\end{multline}
\begin{equation}							\label{eq9002}
\sum_{k=0}^m\lambda^{\frac 1 2-\frac k {2m}}\|D^k  w^{(\varepsilon)}\|_{L_2(Q_{\kappa r}^+(X_0))}
\le N \sum_{|\alpha|\le m}\lambda^{\frac {|\alpha|} {2m}-\frac 1 2}\|f^{(\varepsilon)}_\alpha\|_{L_2(Q_{\kappa r}^+(X_0))}
+ I^{(\varepsilon)}
\end{equation}
for all sufficiently small $\varepsilon$,
where
$$
I^{(\varepsilon)} = N \sum_{|\alpha|\le m}\lambda^{\frac {|\alpha|} {2m}-\frac 1 2} \|f_{\alpha}-f^{(\varepsilon)}_{\alpha}\|_{L_2(\cO_{t_0}^+)}
+ N \sum_{|\alpha|=|\beta|= m}\| (A_{(\varepsilon)}^{\alpha\beta} - A^{\alpha\beta}) D^{\beta} u\|_{L_2(\cO_{t_0}^+)}.
$$
Note that, for the even extension $\cE g$ of a function $g$ defined on $\cO^+_\infty$,
we have
$$
\|\cE g\|_{L_2(Q_r(X_0))}
\le 2\|g\|_{L_2(Q_r^+(X_0))}
\le 2\|\cE g\|_{L_2(Q_r(X_0))}
$$
whenever $X_0 \in \overline{\cO^+_{\infty}}$.
This combined with \eqref{eq9001} and \eqref{eq9002}
gives
\begin{multline*}			
\left(|\cE(D^m w^{(\varepsilon)})|^2\right)_{Q_r(X_0)}^{\frac 1 2}
+ \lambda^{\frac 1 2}  \left(|\cE w^{(\varepsilon)}|^2\right)_{Q_r(X_0)}^{\frac 1 2}
\\
\le N \kappa^{m+\frac d 2} \sum_{|\alpha|\le m}\lambda^{\frac {|\alpha|} {2m}-\frac 1 2}(|\cE f^{(\varepsilon)}_\alpha|^2)_{Q_{\kappa r}(X_0)}^{\frac 1 2}
+ r^{-m-\frac d 2} I^{(\varepsilon)},
\end{multline*}
\begin{multline*}			
\sum_{k=0}^m\lambda^{\frac 1 2-\frac k {2m}}\left(|\cE(D^k w^{(\varepsilon)})|^2\right)_{Q_{\kappa r}(X_0)}^{\frac 1 2}\\
\le N \sum_{|\alpha|\le m}\lambda^{\frac {|\alpha|} {2m}-\frac 1 2}(|\cE  f^{(\varepsilon)}_\alpha|^2)_{Q_{\kappa r}(X_0)}^{\frac 1 2}
+ (\kappa r)^{-m-\frac d 2} I^{(\varepsilon)}.
\end{multline*}
Now by following the corresponding steps in the proof
of Theorem \ref{thm4.2} we see that the left-hand side of the inequality \eqref{eq0917}
is less than the right-hand side of the same inequality
with $f^{(\varepsilon)}_{\alpha}$ in place of $f_{\alpha}$
plus the error term
$$
(r^{-m-\frac d 2} + (\kappa r)^{-m-\frac d 2})I^{(\varepsilon)}.
$$
To finish the proof we let $\varepsilon \searrow 0$.
\end{proof}

\begin{remark}
								\label{rem0824}
Later we need to have the mean oscillation estimate \eqref{eq0917} for all $X_0 \in \cO_{\infty}$,
instead of $X_0 \in \overline{\cO_{\infty}^+}$, for
functions $\cE(D^k u)$, $\cE u$, and $\cE f_{\alpha}$ defined on $\cO_{\infty}$
if the equation \eqref{eq07_08} is satisfied in $\cO_{\infty}^+$.
In order to do this, in case $X_0 \in \cO_{\infty} \setminus \overline{\cO_{\infty}^+}$,
we let $Y_0$ be the reflection point of $X_0$ with respect to the hyper-plane
$\{(t,0,x'): t \in \bR, x' \in \bR^{d-1}\}$.
By Proposition \ref{thm07_02} we get the estimate \eqref{eq0917} with $Y_0$ in place of $X_0$.
Then it is not difficult to see that the estimate \eqref{eq0917} holds true as well for $X_0$ using the evenness of functions involved.
The same claim can be repeated for Corollary \ref{cor8.7}, Proposition \ref{prop3.35}, and Proposition \ref{prop3.36}.
\end{remark}

\subsection{Mean oscillation estimates of $D^{2m}_{x'}u$ for non-divergence type systems on a half space} As a consequence of Proposition \ref{thm07_02}, we easily get

\begin{corollary}
                                \label{cor8.7}
Let $r \in (0,\infty)$, $\kappa \in [64,\infty)$, $\lambda \ge 0$,
and $X_0 = (t_0,x_0) \in \overline{\cO_{\infty}^+}$.
Assume that $u \in W_{2,\text{loc}}^{1,2m}(\cO_{\infty}^+)$
satisfies \eqref{eq06_01} on $\bR \times \bR^{d-1}$ and
$$
u_t + (-1)^m \cL_0 u + \lambda u = f
$$
in $Q_{\kappa r}^+(X_0)$, where $f \in L_{2,\text{loc}}(\cO_{\infty}^+)$.
Then we have
\begin{multline*}
(  |\cE (D_{x'}^{2m} u) - (\cE (D_{x'}^{2m} u))_{Q_r(X_0)}  |)_{Q_r(X_0)}
+ \lambda ( | \cE u - (\cE u)_{Q_r(X_0)}  |)_{Q_r(X_0)}
\\
\le N \kappa^{-\frac 1 2} \sum_{k=0}^{2m} \lambda^{1-\frac{k}{2m}} (|\cE(D^k u)|^2)^{\frac 1 2}_{Q_{\kappa r}(X_0)}
+ N \kappa^{m+\frac d 2} (|\cE f|^2)_{Q_{\kappa r}(X_0)}^{\frac 1 2},
\end{multline*}
where $N = N(d,n,m,\delta)$.
\end{corollary}

\begin{proof}
Since $D_{x'}^m$ satisfies \eqref{eq06_01} on $\bR \times \bR^{d-1}$, we can proceed as in the proof of Corollary \ref{cor06_03}.
\end{proof}

\mysection{Estimates for systems with special coefficients on a half space}
                                        \label{sec8}

The estimates in the previous section imply the $L_p$-estimate of $D_{x'}^m u$ in the divergence case and that of $D_{x'}^{2m}u$ in the non-divergence case. In order to estimate the remaining highest order derivatives, by the interpolation inequality (Proposition \ref{prop07_01}), it suffices to estimate $D_1^{m} u$ in the divergence case and $D_1^{2m}u$ in the non-divergence case. To this end, in this section we consider
$$
\fL_0 u = A(t)D^{2m}_1u + \sum_{j=2}^d D_j^{2m} u,
$$
where $A(t) = A^{\hat{\alpha}\hat{\alpha}}(t)$, $\hat{\alpha}=(m,0, \cdots,0)$.

For this special operator, we have the following improved $L_2$-estimate.
\begin{lemma}
								\label{lem06_04}
Assume that $u \in C_{\text{loc}}^{\infty}(\overline{\cO_{\infty}^+})$
satisfies
\begin{equation}
								\label{eq06_19}
u(t,0,x') = \cdots = D_1^{m-1}u(t,0,x') = 0
\end{equation}
on $Q_R'$
and
\begin{equation}
								\label{eq06_20}
u_t + (-1)^m \fL_0 u = 0	
\end{equation}
in $Q_R^+$.
Then, for any multi-index $\gamma$,
we have
\begin{equation}
								\label{eq06_22}
\| D^\gamma u \|_{L_2(Q_r^+)}+\| D^\gamma u_t \|_{L_2(Q_r^+)}
\le N \| u \|_{L_2(Q_R^+)},	
\end{equation}
where $N = N(d,n,m,\delta,r,R, \gamma)$.	
\end{lemma}

\begin{proof}
As noted in the proof of Corollary \ref{cor33}, it suffices to estimate the first term on the left-hand side of \eqref{eq06_22}.
Also, we only need to treat the case when the multi index $\gamma$ satisfies $\gamma' = 0$,
where $\gamma = (\gamma_1,\gamma')$.
In fact, if the inequality \eqref{eq06_22} is shown to be true with $\gamma' = 0$ and a smaller $R$,
since $D^{\gamma'} u$ satisfies \eqref{eq06_19} on $Q_R'$ and \eqref{eq06_20} in $Q_R^+$,
we can replace $u$ by $D^{\gamma'} u$ in \eqref{eq06_22}.
Then the right-hand side, $N\|D^{\gamma'} u\|_{L_2(Q_R^+)}$, is bounded by that of \eqref{eq06_22}
by Corollary \ref{cor06_01}. Furthermore, by the interpolation inequality with respect to $x_1$,
it suffices to show
\begin{equation}
								\label{eq06_22b}
\| D_1^{2lm} u\|_{L_2(Q_r^+)}
\le N \| u \|_{L_2(Q_R^+)}
\end{equation} for $l=0,1,2,...$.
To prove the above inequality, we first observe that,
thanks to \eqref{eq06_20},
we have
$$
D^{2m}_1u = A^{-1}(t)(-1)^{m+1} u_t - A^{-1}(t)\sum_{j=2}^d D_j^{2m}u
$$
in $Q_{R}^+$.
This together with \eqref{eq06_19} implies that
(first with $l=0$, then inductively)
$$
D^k_1 D^{2ml}_1u(t,0,x') = 0, \quad k = 0, \cdots, m-1,
$$
on $Q_R'$.
Moreover, $D^{2ml}_1u$ satisfies
\eqref{eq06_20} in $Q_R^+$.
Therefore, by Corollary \ref{cor06_01}
applied to $D^{2ml}_1u$
we have
$$
\|D^{2(l+1)m}_1 u \|_{L_2(Q_r^+)}
\le N \|D^{2lm}_1 u\|_{L_2(Q_{r_0}^+)},
$$
where $r < r_0 < R$. This implies \eqref{eq06_22b} by an induction on $l$.
\end{proof}

As a consequence of the previous lemma, we get

\begin{lemma}
								\label{lem06_05}
Let $u \in C_{\text{loc}}^{\infty}(\overline{\cO_{\infty}^+})$
satisfy \eqref{eq06_19}
on $Q'_4$
and \eqref{eq06_20} in $Q_4^+$.
Then, for any multi-index $\gamma$,
$$
\sup_{Q_1^+}|D^\gamma u|+\sup_{Q_1^+}|D^\gamma u_t|
\le N \| u \|_{L_2(Q_4^+)},
$$	
where $N = N(d,n,m,\delta,\gamma)$.
\end{lemma}

\begin{proof}
This is deduced from Lemma \ref{lem06_04} in the same way as Lemma \ref{lemma0611_01} is deduced from Corollary \ref{cor33}.
\end{proof}

Note that in the following H\"{o}lder estimates
the first inequality is for all $D^{\gamma} u$, $|\gamma| = m$,
whereas the second inequality is for $D^{2m}_1 u$ only.
Similarly we see $D^m u$ and $D_1^{2m} u$ in the following lemma and  propositions as well.

\begin{corollary}
								\label{cor07_01}
Let $\lambda \ge 0$, $X_0 = (t_0,0,x_0')$, where $t_0 \in \bR$
and $x'_0 \in \bR^{d-1}$.
Assume that $u \in C_{\text{loc}}^{\infty}(\overline{\cO_{\infty}^+})$
satisfies \eqref{eq06_19}
on $Q_4'(X_0)$ and
\begin{equation}
								\label{eq06_21}
u_t + (-1)^m \fL_0 u + \lambda u = 0	
\end{equation}
in $Q_4^+(X_0)$.
Then
there exists $N = N(d,n,m,\delta)$
such that
\begin{equation}
								\label{eq06_23}
[\cE (D^m u)]_{\cC^{1}(Q_1(X_0))}
\le N\sum_{k=0}^m\lambda^{\frac12-\frac{k}{2m}} \| \cE (D^k u) \|_{L_2(Q_4(X_0))},
\end{equation}
\begin{equation}
								\label{eq06_24}
[\cE (D^{2m}_1 u)]_{\cC^{1}(Q_1(X_0))}
\le N\sum_{k=0}^{2m}\lambda^{1-\frac{k}{2m}} \| \cE (D^k u) \|_{L_2(Q_4(X_0))}.
\end{equation}
\end{corollary}

\begin{proof}
Similar to the proof of Corollary	\ref{cor06_02}, we prove only the case $\lambda = 0$
and $X_0 =(0,0)$.
As noted in the proof of Lemma \ref{lem06_04},
$D_1^{2m}u$ satisfies \eqref{eq06_19}
on $Q'_4$ and \eqref{eq06_20} in $Q_4^+$.
In this case, \eqref{eq06_24} follows immediately from Lemma \ref{lem06_05} applied to $D^{2m}_1 u$.

Lemma \ref{lem06_05} also shows that
$$
[D^m u]_{\cC^{1}(Q_1^+)}
\le N \|u\|_{L_2(Q_4^+)}
\le N \|D^m u\|_{L_2(Q_4^+)},
$$
where the second inequality
is due to the fact that $u$ satisfies \eqref{eq06_19}
and the boundary version of the Poincar\'{e} inequality.
This gives the inequality \eqref{eq06_23}.
\end{proof}

\begin{lemma}
                    \label{lem3.34}
Let $r \in (0,\infty)$, $\kappa \in [64,\infty)$, $\lambda \ge 0$,
and $X_0 = (t_0,x_0) \in \overline{\cO_{\infty}^+}$.
Assume that $u \in C_{\text{loc}}^{\infty}(\overline{\cO_{\infty}^+})$
satisfies \eqref{eq06_19} on $\bR \times \bR^{d-1}$
and \eqref{eq06_21} on $Q_{\kappa r}^+(X_0)$.
Then
$$
(  |\cE (D^m u) - (\cE (D^m u))_{Q_r(X_0)}  |)_{Q_r(X_0)}
\le N \kappa^{-1} \sum_{k=0}^m \lambda^{\frac{1}{2}-\frac{k}{2m}} (|\cE(D^k u)|^2)^{\frac 1 2}_{Q_{\kappa r}(X_0)},
$$
$$
(  |\cE (D_1^{2m} u) - (\cE (D_1^{2m} u))_{Q_r(X_0)}  |)_{Q_r(X_0)}
\le N \kappa^{-1} \sum_{k=0}^{2m} \lambda^{1-\frac{k}{2m}} (|\cE(D^k u)|^2)^{\frac 1 2}_{Q_{\kappa r}(X_0)}.
$$
where $N = N(d,n,m,\delta)$.
\end{lemma}

\begin{proof}
Thanks to the H\"{o}lder estimates in Corollary \ref{cor07_01},
we process as in the proof of Lemma \ref{lem06_03}.	
\end{proof}

From the above lemma, by following the steps in
the proof of Proposition \ref{thm07_02} we prove the following two propositions.

\begin{proposition}
                                    \label{prop3.35}
Let $r \in (0,\infty)$, $\kappa \in [128,\infty)$, $\lambda \ge 0$,
and $X_0 = (t_0,x_0) \in \overline{\cO_{\infty}^+}$.
Assume that $u \in C_{\text{loc}}^{\infty}(\overline{\cO_{\infty}^+})$
satisfies \eqref{eq06_19} on $\bR \times \bR^{d-1}$ and
$$
u_t + (-1)^m \fL_0 u + \lambda u = \sum_{|\alpha|\le m} D^{\alpha}f_{\alpha}
$$
in $Q_{\kappa r}^+(X_0)$,
where $f_{\alpha} \in L_{2,\text{loc}}(\cO_{\infty}^+)$, $|\alpha| \le m$.
Then we have
$$
(  |\cE (D^m u) - (\cE (D^m u))_{Q_r(X_0)}  |)_{Q_r(X_0)}
\le N \kappa^{-1} \sum_{k=0}^m \lambda^{\frac{1}{2}-\frac{k}{2m}} (|\cE(D^k u)|^2)^{\frac 1 2}_{Q_{\kappa r}(X_0)}
$$
$$
+ N \kappa^{m+\frac d 2} \sum_{|\alpha|\le m} \lambda^{\frac{|\alpha|}{2m}-\frac12}
(|\cE f_{\alpha}|^2)_{Q_{\kappa r}(X_0)}^{\frac 1 2},
$$
where $N = N(d,n,m,\delta)$.
\end{proposition}

\begin{proposition}
                            \label{prop3.36}
Let $r \in (0,\infty)$, $\kappa \in [128,\infty)$, $\lambda \ge 0$,
and $X_0 = (t_0,x_0) \in \overline{\cO_{\infty}^+}$.
Assume that $u \in C_{\text{loc}}^{\infty}(\overline{\cO_{\infty}^+})$
satisfies \eqref{eq06_19} on $\bR \times \bR^{d-1}$ and
$$
u_t + (-1)^m \fL_0 u + \lambda u = f
$$
in $Q_{\kappa r}^+(X_0)$, where $f \in L_{2,\text{loc}}(\cO_{\infty}^+)$.
Then we have
$$
(  |\cE (D_1^{2m} u) - (\cE (D_1^{2m} u))_{Q_r(X_0)}  |)_{Q_r(X_0)}
\le N \kappa^{-1} \sum_{k=0}^{2m} \lambda^{1-\frac{k}{2m}} (|\cE(D^k u)|^2)^{\frac 1 2}_{Q_{\kappa r}(X_0)}
$$
$$
+ N \kappa^{m+\frac d 2} (|\cE f|^2)_{Q_{\kappa r}(X_0)}^{\frac 1 2},
$$
where $N = N(d,n,m,\delta)$.
\end{proposition}

\mysection{$L_p$-estimates for systems on a half space}
                    \label{sec9}

With the preparation in the previous two sections, we complete the proofs of Theorem \ref{Thm5} and Theorem \ref{Thm6} in this section.

\begin{proof}[Proof of Theorem \ref{Thm5}]
Recall that the leading coefficients satisfy Assumption \ref{assumption20080424} ($\rho$). As before, we may assume that $T=\infty$, $p>2$, the lower-order coefficients of $\cL$ are all zero, and
$u \in C^{\infty}(\overline{\cO_\infty^+})$
vanishing on $\cO_\infty^+ \setminus Q_{R_0}(X_1)$
for some $X_1\in \overline{\cO_\infty^+}$. 
In this case, it follows from Proposition \ref{thm07_02} (also see Remark \ref{rem0824})
and the proofs of Lemma \ref{lem10.48} as well as
Theorem \ref{Thm1} that
$$
\|D_{x'}^m u\|_{L_p(\cO_\infty^+)}+\lambda^{\frac 1 2}\|u\|_{L_p(\cO_\infty^+)}
\le \|\cE (D_{x'}^m u)\|_{L_p(\cO_\infty)}+\lambda^{\frac 1 2}\|\cE u\|_{L_p(\cO_\infty)}
$$
$$
\le N \kappa_1^{m+\frac d2} \sum_{|\alpha|\le m} \lambda^{\frac{|\alpha|}{2m}-\frac12}
\|\cE f_\alpha\|_{L_p(\cO_\infty)}+N (\kappa_1^{-\frac 12}+\kappa_1^{m+\frac d 2}\rho^{\frac 1 {2\nu}}) \sum_{k=0}^m \lambda^{\frac{1}{2}-\frac{k}{2m}} \|\cE (D^k u)\|_{L_p(\cO_\infty)}
$$
\begin{equation}
                                \label{re2.54}
\le N \kappa_1^{m+\frac d2} \sum_{|\alpha|\le m} \lambda^{\frac{|\alpha|}{2m}-\frac12}
\|f_\alpha\|_{L_p(\cO_\infty^+)}+N (\kappa_1^{-\frac 12}+\kappa_1^{m+\frac d 2}\rho^{\frac 1 {2\nu}}) \sum_{k=0}^m \lambda^{\frac{1}{2}-\frac{k}{2m}} \|D^k u\|_{L_p(\cO_\infty^+)}
\end{equation}
for any $\kappa_1\ge 128$.

Now we move all the spatial derivatives except $D_1^m(A^{\hat\alpha\hat\alpha}D_1^m u)$ to the right-hand side of \eqref{eq081902c}, and add $(-1)^m\sum_{j=2}^d D^{2m}_j u$ to both sides. Here  $\hat\alpha=(m,0,\cdots,0)$.
Then for any $Q_{\kappa_2r}(X_0), \kappa_2\in [128,\infty), r\in (0,\infty),X_0\in \overline{\cO_{\infty}^+}$ and $y\in \bR^d_+$, we have
$$
u_t+(-1)^m(D^m_1(A^{\hat\alpha\hat\alpha}(t,y)D_1^m u)+\sum_{j=2}^d D^{2m}_j u)=\sum_{|\alpha|\le m}D^\alpha\tilde f_\alpha
+ (-1)^m \sum_{j=2}^d D_j^m D_j^m u,
$$
where $\tilde f_\alpha=f_\alpha$ for $|\alpha|<m$,
$$
\tilde f_{\hat\alpha}=f_{\hat\alpha}-\sum_{\substack{|\beta|=m \\ \beta\neq \hat\alpha}}(-1)^m (A^{\hat\alpha\beta}(t,x)+A^{\beta\hat\alpha}(t,y)) D^\beta u
+(-1)^m (A^{\hat\alpha\hat\alpha}(t,y)-A^{\hat\alpha\hat\alpha}(t,x))D^{\hat\alpha}u,
$$
and
$$
\tilde f_\alpha=f_\alpha-\sum_{\substack{|\beta|=m\\\beta\neq \hat\alpha}}(-1)^m A^{\alpha\beta} D^\beta u+(-1)^m (A^{\alpha\hat \alpha}(t,y)-A^{\alpha\hat \alpha}(t,x))D^{\hat \alpha} u
$$
for $|\alpha|=m$, $\alpha\neq \hat \alpha$.
In the last two expressions, we used the fact that
$$
D^{\alpha}_x A^{\alpha\beta}(t,y)D^{\beta}_x u(t,x) = D^{\beta}_x A^{\alpha\beta}(t,y)D^{\alpha}_x u(t,x).
$$
As a consequence of Proposition \ref{prop3.35} and the proof of Lemma \ref{lem10.48}, for any $\kappa_2\ge 128$,
$$
(  |\cE (D^m u) - (\cE (D^m u))_{Q_r(X_0)}  |)_{Q_r(X_0)}
\le N \kappa_2^{-1} \sum_{k=0}^m \lambda^{\frac{1}{2}-\frac{k}{2m}} (|\cE(D^k u)|^2)^{\frac 1 2}_{Q_{\kappa_2 r}(X_0)}
$$
$$
+ N \kappa_2^{m+\frac d 2} \sum_{|\alpha|\le m} \lambda^{\frac{|\alpha|}{2m}-\frac12}
(|\cE f_{\alpha}|^2)_{Q_{\kappa_2 r}(X_0)}^{\frac 1 2}+N\kappa_2^{m+\frac d 2}\rho^{\frac 1 {2\nu}} (|\cE(D_1^m u)|^{2\mu})^{\frac 1 {2\mu}}_{Q_{\kappa_2 r}(X_0)}
$$
$$
+N\kappa_2^{m+\frac d 2}\sum_{|\alpha|=m,\alpha\neq \hat\alpha}(|\cE(D^\alpha u)|^2)^{\frac 1 2}_{Q_{\kappa_2 r}(X_0)}.
$$
Choose a $\mu\in (1,p/2)$. This estimate combined with the Fefferman-Stein theorem and the Hardy-Littlewood maximal function theorem gives
$$
\|D^m u\|_{L_p(\cO_\infty^+)}
\le N \kappa_2^{-1} \sum_{k=0}^m \lambda^{\frac{1}{2}-\frac{k}{2m}}\|D^k u\|_{L_p(\cO_\infty^+)}+N\kappa_2^{m+\frac d 2}\rho^{\frac 1 {2\nu}}\|D_1^m u\|_{L_p(\cO_\infty^+)}
$$
\begin{equation}
                                \label{eq3.40}
+N \kappa_2^{m+\frac d2} \sum_{|\alpha|\le m} \lambda^{\frac{|\alpha|}{2m}-\frac12}
\|f_\alpha\|_{L_p(\cO_\infty^+)}+N\kappa_2^{m+\frac d2}\sum_{|\alpha|=m,\alpha\neq \hat\alpha}\|D^\alpha u\|_{L_p(\cO_\infty^+)}.
\end{equation}
From \eqref{eq3.40} and Proposition \ref{prop07_01}, we get
$$
\|D^m u\|_{L_p(\cO_\infty^+)}
\le N \kappa_2^{-1} \sum_{k=0}^m \lambda^{\frac{1}{2}-\frac{k}{2m}}\|D^k u\|_{L_p(\cO_\infty^+)}+N\kappa_2^{m+\frac d 2}(\rho^{\frac 1 {2\nu}}+\varepsilon)\|D_1^m u\|_{L_p(\cO_\infty^+)}
$$
\begin{equation}
                                \label{eq3.41}
+N \kappa_2^{m+\frac d2} \sum_{|\alpha|\le m} \lambda^{\frac{|\alpha|}{2m}-\frac12}
\|f_\alpha\|_{L_p(\cO_\infty^+)}+N(\varepsilon)\kappa_2^{m+\frac d2}\|D_{x'}^m u\|_{L_p(\cO_\infty^+)}.
\end{equation}
Combining \eqref{re2.54} and \eqref{eq3.41} we obtain the desired estimate by first taking $\kappa_2$ sufficiently large, then $\varepsilon$ sufficiently small, $\kappa_1$ sufficiently large, and finally $\rho$ sufficiently small.
\end{proof}

\begin{proof}[Proof of Theorem \ref{Thm6}]
It suffices to establish the apriori estimate when $T=\infty$, the lower-order coefficients of $L$ are all zero, and
$u \in C^{\infty}(\overline{\cO_\infty^+})$
vanishes on $\cO_\infty^+ \setminus Q_{R_0}(X_1)$
for some $X_1\in \overline{\cO_\infty^+}$. We use the strategy in the proof of Theorem \ref{Thm5} and consider two cases.

{\em Case 1: $p\in (2,\infty)$.} It follows from Corollary \ref{cor8.7} that
$$
\|D_{x'}^{2m} u\|_{L_p(\cO_\infty^+)}+\lambda\|u\|_{L_p(\cO_\infty^+)}
\le N \kappa_1^{m+\frac d2} \|f\|_{L_p(\cO_\infty^+)}
$$
\begin{equation}
                                \label{eq5.31}
+N (\kappa_1^{-\frac 12}+\kappa_1^{m+\frac d 2}\rho^{\frac 1 {2\nu}}) \sum_{k=0}^m \lambda^{1-\frac{k}{2m}} \|D^k u\|_{L_p(\cO_\infty^+)}
\end{equation}
for any $\kappa_1\ge 64$. We move all the spatial derivatives except $A^{\hat\alpha\hat\alpha} D_1^{2m} u$ to the right-hand side of \eqref{eq081902d}, and add $(-1)^m\sum_{j=2}^d D^{2m}_j u$ to both sides.
As a consequence of Proposition \ref{prop3.36} and the proof of Lemma \ref{lem10.48}, for any $\kappa_2\ge 128$,
$$
(  |\cE (D_1^{2m} u) - (\cE (D_1^{2m} u))_{Q_r(X_0)}  |)_{Q_r(X_0)}
\le N \kappa_2^{-1} \sum_{k=0}^{2m} \lambda^{1-\frac{k}{2m}} (|\cE(D^k u)|^2)^{\frac 1 2}_{Q_{\kappa_2 r}(X_0)}
$$
$$
+ N \kappa_2^{m+\frac d 2}
(|\cE f|^2)_{Q_{\kappa_2 r}(X_0)}^{\frac 1 2}+N\kappa_2^{m+\frac d 2}\rho^{\frac 1 {2\nu}} (|\cE(D_1^{2m} u)|^{2\mu})^{\frac 1 {2\mu}}_{Q_{\kappa_2 r}(X_0)}
$$
$$
+N\kappa_2^{m+\frac d 2}\sum_{|\alpha|=2m,\alpha\neq 2\hat\alpha}(|\cE(D^\alpha u)|^2)^{\frac 1 2}_{Q_{\kappa_2 r}(X_0)}.
$$
This estimate combined with the Fefferman-Stein theorem and the Hardy-Littlewood maximal function theorem gives
$$
\|D_1^{2m} u\|_{L_p(\cO_\infty^+)}
\le N \kappa_2^{-1} \sum_{k=0}^{2m} \lambda^{1-\frac{k}{2m}}\|D^k u\|_{L_p(\cO_\infty^+)}+N\kappa_2^{m+\frac d 2}\rho^{\frac 1 {2\nu}}\|D_1^{2m} u\|_{L_p(\cO_\infty^+)}
$$
\begin{equation}
                                \label{eq3.40b}
+N \kappa_2^{m+\frac d2} \|f\|_{L_p(\cO_\infty^+)}+N\kappa_2^{m+\frac d2}\sum_{|\alpha|=2m,\alpha\neq 2\hat\alpha}\|D^\alpha u\|_{L_p(\cO_\infty^+)}.
\end{equation}
From \eqref{eq3.40b} and Proposition \ref{prop07_01}, we get
$$
\|D^{2m} u\|_{L_p(\cO_\infty^+)}
\le N \kappa_2^{-1} \sum_{k=0}^{2m} \lambda^{1-\frac{k}{2m}}\|D^k u\|_{L_p(\cO_\infty^+)}+N\kappa_2^{m+\frac d 2}(\rho^{\frac 1 {2\nu}}+\varepsilon)\|D_1^{2m} u\|_{L_p(\cO_\infty^+)}
$$
\begin{equation}
                                \label{eq3.41b}
+N \kappa_2^{m+\frac d2} \|f\|_{L_p(\cO_\infty^+)}+N(\varepsilon)\kappa_2^{m+\frac d2}\|D_{x'}^{2m} u\|_{L_p(\cO_\infty^+)}.
\end{equation}
Combining \eqref{eq5.31} and \eqref{eq3.41b} we obtain the desired estimate by first taking $\kappa_2$ sufficiently large, then $\varepsilon$ sufficiently small, $\kappa_1$ sufficiently large, and finally $\rho$ sufficiently small.

{\em Case 2: $p\in (1,2]$.} Thanks to Case 1 and Remark \ref{rem2.01}, we already have the $W^{1,2m}_q$ solvability of
$$
u_t+(-1)^m \cL_0 u+\lambda u=f
$$
on the half space for any $q\in (2,\infty)$ and $\lambda>0$.
The same duality argument in the proof of Theorem \ref{Thm2} yields the solvability of the same equation for any $q\in (1,2)$. We can repeat the argument in Section \ref{sec8} to deduce a version of Proposition \ref{prop3.36} with $2$ norms replaced by $q$ norms. Inspecting the proof of Case 1, to finish the proof it remains to have a proper version of Corollary \ref{cor8.7} with $2$ norms replaced by $q$ norms.

We claim that Lemma \ref{lem06_02} is still true with $L_2$ replaced by $L_q,q\in (1,\infty)$, i.e., if $u \in C_{\text{loc}}^{\infty}(\overline{\cO_{\infty}^+})$
satisfies \eqref{eq06_01} on $Q'_4$
and \eqref{eq06_09} in $Q_4^+$, then
$$
[u]_{\cC^{1/2}(Q_1^+)}
\le N \| u \|_{L_q(Q_4^+)}.
$$
This easily yields the desired version of Corollary \ref{cor8.7} by following the lines in Section \ref{sec7}. However, the claim does not follow directly from the proof of Lemma \ref{lem06_02} because \eqref{eq06_16} doesn't hold if the $W^{1,2}_2$ norm on the right-hand side is replace by the $W^{1,2}_q$ norm when $q$ is close to $1$. To get around this, we use a bootstrap argument. We first note that under the assumption of Lemma \ref{lem06_02}, for any $1<r<R\le 4$, it holds that
\begin{equation}
                        \label{eq11.25}
\|u\|_{W_q^{1,2m}(Q_r^+)}\le N\|u\|_{L_q(Q_R^+)}.
\end{equation}
This can be shown in the same way as Lemma \ref{lem6.2} and \ref{lem06_01} based on the global $W_q^{1,2m}$ estimate on the half space. By the Sobolev imbedding theorem and \eqref{eq11.25}, we have
\begin{equation*}
\|u\|_{L_{q_1}(Q_r^+)}\le N\|u\|_{L_q(Q_R^+)}
\end{equation*}
for any $q_1>q$ satisfying
$$
\frac 1 {q_1}>\frac 1 q-\frac 1 {d+1}.
$$
We iterate this bootstrap process for a finite many steps on a sequence of shrinking half cylinders, and get
\begin{equation*}
\|u\|_{W_{q_l}^{1,2m}(Q_1^+)}\le N\|u\|_{L_q(Q_4^+)},
\end{equation*}
where $q_l>2(d+1)$. Now by the Sobolev imbedding theorem again, we deduce
\begin{equation*}
\|u\|_{\cC^{1/2}(Q_1^+)}\le N\|u\|_{L_q(Q_4^+)},
\end{equation*}
which is exactly our claim. The theorem is proved.
\end{proof}

\begin{remark}
From the bootstrap argument above, we actually can get a finer boundary estimate as follows. If $u \in W_{q,\text{loc}}^{1,2m}(\overline{\cO_{\infty}^+}),q\in (1,\infty)$
satisfies \eqref{eq06_01} on $Q'_4$
and \eqref{eq06_09} in $Q_4^+$, then for any  and $\varepsilon\in (0,1)$,
$$
[u]_{\cC^{1-\varepsilon,2m-\varepsilon}(Q_1^+)}
\le N \| u \|_{L_q(Q_4^+)},
$$
where $N=N(d,m,n,q,\varepsilon)$.
\end{remark}

\mysection{Systems on a bounded domain}         \label{sec10}

We present the proofs of Theorem \ref{Thm7} and \ref{Thm8} in this section. We first treat the non-divergence systems. In this case, the proof is quite standard by using the technique of flattening the boundary and a partition of the unity. We give a sketched proof for the sake of completeness.

\begin{proof}[Proof of Theorem \ref{Thm8}]
First, in a same way as Lemma \ref{lem06_01} by using Theorem \ref{Thm2} instead of Theorem \ref{th06_05}, we obtain the following interior estimate for any $0<r<R<\infty$, $Q_r\subset Q_R\subset \Omega_T$ and $\lambda\ge \lambda_0$
\begin{equation}							 \label{eq7.5.08}
\| u_t\|_{L_p(Q_r)}+\sum_{|\alpha|\le 2m} \lambda^{1-\frac{|\alpha|}{2m}} \| D^{\alpha} u \|_{L_p(Q_r)}
\le N \| f\|_{L_p(Q_R)}+N \| u\|_{L_p(Q_R)}.
\end{equation}

Similarly, Theorem \ref{Thm6} yields a boundary estimate: let $0<r<R<\infty$, $f\in L_p(Q_R^+)$, and $\rho$ be the constant taken from Theorem \ref{Thm6}. Then under Assumption \ref{assumption20080424} ($\rho$),
for any $\lambda\ge \lambda_0$ and $u \in W_p^{1,2m}(Q_R^+)$,
we have
\begin{equation}							 \label{eq7.5.09}
\| u_t\|_{L_p(Q_r^+)}+\sum_{|\alpha|\le 2m} \lambda^{1-\frac{|\alpha|}{2m}} \| D^{\alpha} u \|_{L_p(Q_r^+)}
\le N \| f\|_{L_p(Q_R^+)}+N \| u\|_{L_p(Q_R^+)},	
\end{equation}
provided that $u=D_1u=...=D_1^{m-1}u=0$ on $Q_R'$ and
\begin{equation*}		
u_t+(-1)^m L u + \lambda u = f \quad \text{in}\,\,Q_R^+.
\end{equation*}

It is well-known that the ellipticity condition \eqref{eq7.9.17} is preserved under a change of variables. Take $t_0\in (-\infty,T)$, a point $x_0\in \partial \Omega$ and a number $r_0=r_0(\Omega)$, so that
$$
\Omega\cap B_{r_0}(x_0) = \{x \in  B_{r_0}(x_0)\, :\, x_1 >\phi(x')\}
$$
in some coordinate system.
We now locally flatten the boundary of $\partial\Omega$ by defining
$$
y_1=x_1-\phi(x'):=\Phi^1(x),\quad y_j=x_j:=\Phi^j(x),\,\,j\ge 2.
$$
Under the assumptions of the theorem, $\Phi$ is a $C^{2m-1,1}$ diffeomorphism in a neighborhood of $x_0$. It is easily seen that the leading coefficients of the new operator in the $y$-coordinates also satisfy Assumption \ref{assumption20080424} with a possibly different $\rho$. Thus, we can choose a sufficiently small $\rho$ such that from \eqref{eq7.5.09}, for $X_0 = (t_0,x_0)$ and some $r_1=r_1(\Omega)<r_0$,
$$
\| u_t\|_{L_p(\Omega_T\cap Q_{r_1}(X_0))}+\sum_{|\alpha|\le 2m} \lambda^{1-\frac{|\alpha|}{2m}} \| D^{\alpha} u \|_{L_p(\Omega_T\cap Q_{r_1}(X_0))}
$$
\begin{equation}							 \label{eq7.5.31}
\le N \| f\|_{L_p(\Omega_T\cap Q_{r_0}(X_0))}+N\sum_{j=0}^{m-1}\|D^j u\|_{L_p(\Omega_T\cap Q_{r_0}(X_0))}.
\end{equation}
Finally, a partition of the unity together with \eqref{eq7.5.08} and \eqref{eq7.5.31} completes the proof for a sufficiently large $\lambda_0$.
\end{proof}

Now we turn to the divergence case. We need to introduce a special mollification, which was used, for instance, in \cite{GH90,Mi06}.

\begin{proof}[Proof of Theorem \ref{Thm7}]
Again we only give an outline of the proof. The interior estimate is similar to that of the non-divergence case. Theorem \ref{Thm1} implies that, for any $0<r<R<\infty$, $Q_r\subset Q_R\subset \Omega_T$ and $\lambda\ge \lambda_0$,
\begin{equation*}							
\sum_{|\alpha|\le m} \lambda^{1-\frac{|\alpha|}{2m}} \| D^{\alpha} u \|_{L_p(Q_r)}
\le N \sum_{|\alpha| \le m} \lambda^{\frac{|\alpha|}{2m}} \| f_{\alpha} \|_{L_p(Q_R)}+N \| u \|_{L_p(Q_R)}.
\end{equation*}
We also have the boundary estimate by Theorem \ref{Thm5}: Let $0<r<R<\infty$, $f\in L_p(Q_R^+)$, and $\rho$ be the constant taken from Theorem \ref{Thm5}. Then under Assumption \ref{assumption20080424} ($\rho$),
for any $\lambda\ge \lambda_0$ and $u \in \cH_p^{2m}(Q_R^+)$,
we have
\begin{equation}							 \label{eq8.2.36}
\sum_{|\alpha|\le m} \lambda^{1-\frac{|\alpha|}{2m}} \| D^{\alpha} u \|_{L_p(Q_r^+)}
\le N \sum_{|\alpha| \le m} \lambda^{\frac{|\alpha|}{2m}} \| f_{\alpha} \|_{L_p(Q_R^+)}
+N\|u \|_{L_p(Q_R^+)},	
\end{equation}
provided that $u=D_1u=...=D_1^{m-1}u=0$ on $Q_R'$ and
\begin{equation*}		
u_t+(-1)^m \cL u + \lambda u = f \quad \text{in}\,\,Q_R^+.
\end{equation*}

Take $t_0\in (-\infty,T)$, a point $x_0\in \partial \Omega$ and a number $r_0\in (0,R_1]$. By Assumption \ref{assump2}, locally in some coordinate system, we have
$$
\Omega\cap B_{r_0}(x_0) = \{x \in  B_{r_0}(x_0)\, :\, x^1 >\phi(x')\},
$$
and the local Lipschitz norm of $\phi$ is less than $\rho_1$.
The goal is to locally flatten the boundary of $\partial\Omega$. However, $\phi$ is not smooth in this case since it is only assumed to be Lipschitz continuous. To construct a smooth  diffeomorphism, we define a function $\tilde \phi$ on $\bR_+^d$ by
$$
\tilde \phi(x)=\int_{\bR^{d-1}}\eta(y')\phi(x'- x_1 y')\,dy'.
$$
Here $\eta\in C_0^\infty(B_1')$ has unit integral. It is easy to check that $\tilde \phi(0,x')=\phi(x')$ and $|D^k \tilde \phi(x)|\le N(x_1)^{1-k}\rho_1$.
We now  define
$$
y_1=x_1-\tilde \phi(x):=\tilde \Phi^1(x),\quad y_j=x_j:=\tilde \Phi^j(x),\,\,j\ge 2.
$$
As before, the leading coefficients of the new operator in the $y$-coordinates also satisfy Assumption \ref{assumption20080424} with a possibly different $\rho$. After some straightforward calculations using \eqref{eq7.5.09} and Hardy's inequality, we conclude, for $X_0 = (t_0,x_0)$ and some $r_1=r_1(\Omega) \in (0,r_0)$,
$$
\sum_{|\alpha|\le m} \lambda^{1-\frac{|\alpha|}{2m}} \| D^{\alpha} u \|_{L_p(\Omega_T\cap Q_{r_1}(X_0))}
\le N \sum_{|\alpha|\le m}\lambda^{\frac {|\alpha|} {2m}}\| f_\alpha\|_{L_p(\Omega_T\cap Q_{r_0}(X_0))}
$$
\begin{equation}							 \label{eq8.5.23}
+N\rho_1\sum_{|\alpha|\le m} \lambda^{1-\frac{|\alpha|}{2m}} \| D^{\alpha} u \|_{L_p(\Omega_T\cap Q_{r_0}(X_0))}.
\end{equation}
Using a partition of the unity together with \eqref{eq8.2.36} and \eqref{eq8.5.23}, we complete the proof of the theorem upon choosing a sufficiently large $\lambda_0$ and small $\rho_1$.
\end{proof}

\mysection{Remarks on the ellipticity conditions}
							\label{sec11}

In this section we discuss some other ellipticity conditions appeared in the literature, and show how our results can be extended to systems under those conditions.

The following strong ellipticity condition has been widely used before; see, for example, \cite{LKS,B09}.
\begin{assumption}
For all $(t,x) \in \bR^{d+1}$ and complex vectors $\xi=\{\xi_{\alpha,i}\},|\alpha|=m,i=1,...,n$,
\begin{equation}
                                \label{eq9.9.17}
\Re\left(\sum_{|\alpha|=|\beta|=m} \xi_{\alpha,i}\overline{\xi_{\beta,j}}A^{\alpha\beta}_{ij}(t,x)\right)
\ge \delta |\xi|^{2},
\end{equation}
where $\delta > 0$.
\end{assumption}

The next condition is called uniform parabolicity in the sense of Petrovskii, which has been used, for example, in \cite{Solo,LSU,Ejd,PS1}. We define a matrix-valued function on $\bR^{d+1}\times (\bR^d \setminus \{0\})$:
$$
\bfA(t,x,\xi)=|\xi|^{-2m}\sum_{|\alpha|=|\beta|=m} \xi^{\alpha}\xi^{\beta} A^{\alpha\beta}(t,x).
$$
\begin{assumption}
Let $\lambda_j(t,x,\xi)$, $j=1,...,n$, be the eigenvalues of $\bfA(t,x,\xi)$. Then,
\begin{equation}
                                \label{eq14.34}
\Re \left(\lambda_j(t,x,\xi)\right) \ge \delta,\quad j=1,2,...,n,
\end{equation}
for all $(t,x) \in \bR^{d+1}$ and $\xi \in \bR^d\setminus \{0\}$, where $\delta > 0$.
\end{assumption}

We still assume that all the coefficients are bounded and measurable. Clearly, the Legendre-Hadamard ellipticity condition \eqref{eq7.9.17} is weaker than the strong ellipticity condition. However, it is stronger than the uniform parabolicity in the sense of Petrovskii.

\subsection{The strong ellipticity condition} Since it is stronger than our assumption, all the results in this paper hold true under this condition.
Moreover, we can take $\lambda_0=0$ in Theorem \ref{Thm7} for divergence form parabolic systems without lower-order terms.
In this case the solution $u$ satisfies
\begin{equation}							 \label{eq9.5.24}
\sum_{|\alpha|\le m} \| D^{\alpha} u \|_{L_p(\Omega_T)}
\le N \sum_{|\alpha| \le m}\| f_{\alpha} \|_{L_p(\Omega_T)}.
\end{equation}
Indeed, by the method of continuity it suffices to prove the estimate \eqref{eq9.5.24}. Due to \eqref{eq9.9.17} and the Poincar\'{e} inequality, we easily get the unique solvability  for $p=2$ as well as
\begin{equation}							 \label{eq9.5.24b}
\sum_{|\alpha|\le m} \| D^{\alpha} u \|_{L_2(\Omega_T)}
\le N \sum_{|\alpha| \le m}\| f_{\alpha} \|_{L_2(\Omega_T)}.
\end{equation}
In the case when $p>2$, we add $(\lambda_0+1) u$ to both sides of the first equation of \eqref{eq8.10.51}. By Theorem \ref{Thm7}, it holds that
\begin{equation}							 \label{eq9.5.25}
\sum_{|\alpha|\le m} \| D^{\alpha} u \|_{L_p(\Omega_T)}
\le N_1 \sum_{|\alpha| \le m}\| f_{\alpha} \|_{L_p(\Omega_T)}+N_1\|u\|_{L_p(\Omega_T)}.
\end{equation}
Take $p_1\in (p,\infty)$ such that  $1-d/p>-d/p_1$. By H\"older's inequality, Young's inequality and the Poincar\'e-Sobolev inequality, we get for any $\varepsilon>0$,
$$
\|u\|_{L_p(\Omega_T)}\le N(\varepsilon)\|u\|_{L_2(\Omega_T)}+ \varepsilon\|u\|_{L_{p_1}(\Omega_T)}\le N(\varepsilon)\|u\|_{L_2(\Omega_T)}+N_2\varepsilon\|Du\|_{L_{p}(\Omega_T)}.
$$
Choosing $\varepsilon=1/(2N_1N_2)$ and using \eqref{eq9.5.24b} and \eqref{eq9.5.25}, we obtain \eqref{eq9.5.24} for $p>2$. The remaining case $p\in (1,2)$ follows from the standard duality argument.

\subsection{The uniform parabolicity condition in the sense of Petrovskii}

As we noted, this assumption is weaker than the Legendre-Hadamard condition. Under this assumption, for the solvability of parabolic systems, we need to impose a stronger regularity assumption on the leading coefficient, that is, they are VMO in both $t$ and $x$. More precisely, set
$$
\text{osc}_{t,x}\left(A^{\alpha\beta},Q_r(t,x)\right)
= \dashint_{Q_r(t,x)} \big| A^{\alpha\beta}(s,y) - \dashint_{Q_r(t,x)} A^{\alpha\beta}\big| \, dy \, ds,
$$
and
$$
\tilde A^{\#}_R = \sup_{(t,x) \in \bR^{d+1}} \sup_{r \le R}  \sup_{|\alpha|=|\beta|=m} \text{osc}_{x} \left(A^{\alpha\beta},Q_r(t,x)\right).
$$
\begin{assumption}[$\rho$]                          \label{assumption20090718}
There is a constant $R_0\in (0,1]$ such that $\tilde A_{R_0}^{\#} \le \rho$.
\end{assumption}

Next we show that the results in Section \ref{sec082001} about parabolic systems in the whole space (Theorem \ref{Thm1} and \ref{Thm2}) still hold true under the assumptions above. As a consequence, we obtain interior estimates for both divergence and non-divergence type parabolic systems. We note that, for non-divergence type parabolic systems, the corresponding interior estimate was established in a recent interesting paper \cite{PS1} (see Theorem 2.4 there) by using a completely different approach.

By inspecting the proofs of the main theorems, it is apparent that if the $L_2$-estimate Theorem \ref{thVMO01} 
is proved for parabolic systems with {\em constant} coefficients under the uniform parabolicity condition, then the remaining arguments can be carried out as before with obvious modifications. Indeed, we have

\begin{theorem}							\label{thVMO09}
Let $T \in (-\infty,\infty]$ and
$$
\cL_0 u = \sum_{|\alpha|=|\beta|=m}D^{\alpha}(A^{\alpha\beta}D^{\beta} u),
$$
where $A^{\alpha\beta}$ are constants satisfying the uniform parabolicity condition \eqref{eq14.34}. Then there exists $N = N(d,n,m,\delta)$ such that,
for any $\lambda \ge 0$,
\begin{equation}							\label{eqVMO03k}
\sum_{|\alpha|\le m} \lambda^{1-\frac{|\alpha|}{2m}} \| D^{\alpha} u \|_{L_2(\cO_T)}
\le N \sum_{|\alpha| \le m} \lambda^{\frac{|\alpha|}{2m}} \| f_{\alpha} \|_{L_2(\cO_T)},	
\end{equation}
if $u \in \cH_2^m(\cO_T)$, $f_{\alpha} \in L_2(\cO_T)$, $|\alpha|\le m$, and
\begin{equation}							\label{eqVMO01k}
u_t + (-1)^m \cL_0 u + \lambda u = \sum_{|\alpha|\le m} D^{\alpha} f_{\alpha}
\end{equation}
in $\cO_T$.
Furthermore, for $\lambda > 0$ and $f_{\alpha} \in L_2(\cO_T)$, $|\alpha| \le m$,
there exists a unique $u \in \cH_2^m(\cO_T)$ satisfying \eqref{eqVMO01k}.
\end{theorem}

Theorem 11.4 is probably known before. For example, it can be derived from the results in \cite{Solo}; see also Theorem 10.4 in Chapter VII of \cite{LSU}. Instead of appealing to those general results, here we present a direct proof of it.
We need an elementary lemma, which is verified by a direction computation.
\begin{lemma}
                                    \label{lem16.07}
Let $\delta>0$ and $U$ be an $n\times n$ upper triangular complex matrix satisfying
$$
|U|\le \delta^{-1},\quad \Re\lambda_j\ge \delta,\quad j=1,2,...,n,
$$
where $\lambda_j$, $j=1,...,n$, are the eigenvalues of $U$.
Then there exist real constants $\varepsilon,\delta_1>0$, depending  only on $n$ and $\delta$, such that for any $x\in \bC^n$
$$
\Re (x^H BU x)\ge \delta_1 |x|^2,
$$
where $B=\text{diag}\{\varepsilon^{n-1},\varepsilon^{n-2},...,\varepsilon,1\}$ and $x^H$ is the conjugate transpose of $x$.
\end{lemma}

\begin{proof}[Proof of Theorem \ref{thVMO09}]
It suffices to prove \eqref{eqVMO03k} when $u \in C_0^{\infty}(\overline{\cO_T})$ and $\lambda > 0$. We take the Fourier transform of \eqref{eqVMO01k} in $x$ and get
\begin{equation}
                                    \label{eq11.05}
\tilde u_t + \bfA(\xi) |\xi|^{2m}\tilde u + \lambda \tilde u = \sum_{|\alpha|\le m} (\ii \xi)^\alpha \tilde f_{\alpha}.
\end{equation}
Let $\bfA(\xi)=Q^H U Q$ be the Schur decomposition of $\bfA$, where $Q=Q(\xi)$ is an $n\times n$ unitary matrix and $U=U(\xi)$
is an upper triangular matrix.
Let $B$ be the diagonal matrix in Lemma \ref{lem16.07}. Multiplying both sides of \eqref{eq11.05} by $Q^H B Q\tilde u$ and integrating on $\cO_T$ give us
$$
\langle B Q\tilde u,Q\tilde u_t\rangle_{\cO_T} + \langle B Q\tilde u,U Q |\xi|^{2m} \tilde  u\rangle_{\cO_T} + \lambda \langle B Q\tilde u,Q\tilde u\rangle_{\cO_T}
$$
\begin{equation}
                                    \label{eq11.40}
= \sum_{|\alpha|\le m} \langle Q^H B Q\tilde u,(\ii \xi)^\alpha \tilde f_{\alpha}\rangle_{\cO_T}.
\end{equation}
As in the proof of Theorem \ref{thVMO01},
$$
\Re \langle B Q\tilde u,Q\tilde u_t\rangle_{\cO_T} \ge 0.
$$
By the Plancherel equality,
$$
\quad \lambda \Re \langle B Q\tilde u,Q\tilde u\rangle_{\cO_T}\ge N(\varepsilon)\lambda \|u\|^2_{L_2(\cO_T)}.
$$
To estimate the second term of the left-hand side of \eqref{eq11.40}, we use Lemma \ref{lem16.07} and the Plancherel equality to get
$$
\Re \langle B Q\tilde u,U Q |\xi|^{2m} \tilde  u\rangle_{\cO_T}\ge
\delta_1\langle \tilde u,|\xi|^{2m} \tilde  u\rangle_{\cO_T}
\ge N(n,m,\delta)\|D^m u\|^2_{L_2(\cO_T)}.
$$
The real part of the right-hand side of \eqref{eq11.40} is bounded from above by
\begin{multline*}
N\sum_{|\alpha|\le m}\|D^{\alpha}u\|_{L_2(\cO_T)}\|f_{\alpha}\|_{L_2(\cO_T)}\\ \le \sum_{|\alpha|\le m}\varepsilon \lambda^{\frac {m-|\alpha|} m}\|D^\alpha u\|^2_{L_2(\cO_T)} + N\sum_{|\alpha|\le m}\varepsilon^{-1} \lambda^{-\frac {m-|\alpha|} m}\|f_\alpha\|^2_{L_2(\cO_T)}
\end{multline*}
for all $\varepsilon > 0$. To complete the proof of \eqref{eqVMO03k} it suffices to use the interpolation inequalities and choose an appropriate $\varepsilon$.
\end{proof}

\begin{remark}
In contrast, under Petrovskii's parabolicity condition, the Dirichlet boundary value problem of parabolic systems is in general not well-posed when $d\ge 2$, as pointed out in \S 10 Chapter VII of \cite{LSU}. However, in the case $d=1$, relying on a linear transformation one can extend Theorem \ref{thVMO09} to systems on the half space with the homogeneous Dirichlet boundary condition; see, for instance, \S 10 Chapter VII of \cite{LSU}. Thus, all the results in Section \ref{sec082001} about systems on a half space or a bounded domain remain true in this case.
\end{remark}

\section*{Acknowledgement}
The authors are very grateful to Nicolai V. Krylov and the referees for helpful
comments on the first version of the paper.

\bibliographystyle{plain}

\def\cprime{$'$}\def\cprime{$'$} \def\cprime{$'$} \def\cprime{$'$}
  \def\cprime{$'$} \def\cprime{$'$}

\end{document}